\documentclass[onefignum,onetabnum]{siamonline250211}

\usepackage{graphicx} 
\usepackage{amsmath,amssymb}
\usepackage{mathrsfs}
\usepackage{enumitem}
\usepackage{xcolor}
\usepackage{subcaption}
\usepackage{tikz,tikzscale}
\usepackage{pgfplots}
\pgfplotsset{compat=1.13}
\usepackage{listings}

\newtheorem{assumption}{Assumption}

\usepackage{lipsum}
\usepackage{amsfonts}
\usepackage{graphicx}
\usepackage{epstopdf}
\usepackage{algorithmic}
\ifpdf
  \DeclareGraphicsExtensions{.eps,.pdf,.png,.jpg}
\else
  \DeclareGraphicsExtensions{.eps}
\fi

\usepackage{enumitem}
\setlist[enumerate]{leftmargin=.5in}
\setlist[itemize]{leftmargin=.5in}

\newsiamremark{remark}{Remark}
\newsiamremark{hypothesis}{Hypothesis}
\crefname{hypothesis}{Hypothesis}{Hypotheses}
\newsiamthm{claim}{Claim}
\newsiamremark{example}{Example}
  
\headers{FEM for parametric PDEs}{V. Kaarnioja, A.~Rupp, J.~Gopalakrishnan}

\title{Sufficient conditions for QMC analysis of finite elements for\\parametric
  differential equations%
  \thanks{\today.\\
  \corr{Supported by the Research Council of Finland (Flagship of Advanced Mathematics for Sensing, Imaging and Modelling grant 359183), by the Deutsche Forschungsgemeinschaft (DFG, German Research Foundation) -- 577175348, and by the US National Science Foundation under grant DMS-2245077.}}%
  }

\author{Vesa Kaarnioja\footnotemark[2]\thanks{School of Engineering Sciences, LUT University, PO~Box 20, 53851 Lappeenranta, Finland (\email{vesa.kaarnioja@lut.fi}).} \and Andreas Rupp\footnotemark[3]\thanks{Department of Mathematics, Saarland University, 66123 Saarbrücken, Germany (\email{andreas.rupp@uni-saarland.de}).} \and Jay Gopalakrishnan\footnotemark[4]\thanks{Portland State University, PO Box 751, Portland OR 97207, USA (\email{gjay@pdx.edu}).}
}

\usepackage{amsopn}

\ifpdf
\hypersetup{
  pdftitle={Sufficient conditions for QMC analysis of finite elements for parametric
  differential equations},
  pdfauthor={J.~Gopalakrishnan, V. Kaarnioja, A.~Rupp}
}
\fi

\newcommand{\EE}{{\mathbb{E}}}
\newcommand{\dy}{\textup{d}\bd y}
\def\d{\partial}
\newcommand{\vphi}{\varphi}
\newcommand{\og}{\omega}
\newcommand{\Qmc}{{Q_{\bd{\Delta}}^\varphi}}
\newcommand*{\mesh}{\ensuremath{{\mathcal E_h}}}
\newcommand*{\faces}{\ensuremath{\mathcal F_h}}
\newcommand*{\elem}{\ensuremath{E}}

\newcommand*{\normal}{\ensuremath{\vec n}}
\newcommand*{\dx}{\ensuremath{\; \textup d\vec x}}
\newcommand*{\ds}{\ensuremath{\; \textup d\sigma}}
\newcommand*{\dive}{\ensuremath{\operatorname{div}}}
\newcommand*{\setu}{\mathfrak u}
\newcommand*{\setw}{\mathfrak w}
\newcommand*{\setv}{\mathfrak v}

\newcommand{\IR}{\ensuremath{\mathbb R}}
\newcommand{\IN}{{\ensuremath{\mathbb N}}}

\newcommand*{\der}{\ensuremath{\partial^{\vec \nu}_{\vec y}}}
\newcommand*{\derm}{\ensuremath{\partial^{\vec m}_{\vec y}}}
\newcommand*{\dernum}{\ensuremath{\partial^{\vec \nu - \vec m}_{\vec y}}}
\newcommand*{\amax}[1][]{\ensuremath{\overline a^{#1}(\vec y)}}

\newcommand*{\bsnu}{\ensuremath{\vec \nu}}
\newcommand*{\bsy}{\ensuremath{\vec y}}
\newcommand*{\bsx}{\ensuremath{\vec x}}

\newcommand{\cl}[1]{{\mathcal{{#1}}}}
\newcommand{\bb}[1]{{\mathbb{{#1}}}}
\newcommand{\bd}[1]{{\boldsymbol{{#1}}}}
\renewcommand{\vec}{\boldsymbol}

\newcommand*{\corr}[1]{{#1}}

\usepackage{enumitem}
\usepackage{microtype}
\begin{document}

\maketitle

\begin{abstract}
 Parametric regularity of discretizations of flux vector fields satisfying a balance law is studied under some assumptions on a random parameter that links the flux with an unknown primal variable (often through a constitutive law).  In the primary example of the stationary diffusion equation, the parameter corresponds to the inverse of the diffusivity.  The random parameter is modeled here as a Gevrey-regular random field. Specific focus is on random fields expressible as functions of countably infinite sequences of independent random variables, which may be uniformly or normally distributed.  Quasi-Monte Carlo (QMC) error bounds for some quantity of interest that depends on the flux are then derived using the parametric regularity. It is shown that the QMC method \corr{achieves a dimension-independent, faster-than-Monte Carlo convergence rate} if the quantity of interest depends continuously  on the primal variable, its flux, or its gradient. A series of assumptions are introduced with the goal of encompassing a broad class of discretizations by various finite element methods.  The assumptions are verified for the diffusion equation discretized using conforming finite elements, mixed methods, and hybridizable discontinuous Galerkin schemes.  Numerical experiments confirm the analytical findings, highlighting the role of accurate flux approximation in QMC methods.
\end{abstract}

\begin{keywords}
 mixed method, discontinuous Galerkin, quasi-Monte Carlo, random coefficient, Gevrey regularity, partial differential equation
\end{keywords}

\begin{AMS}
 65C05, 65N30
\end{AMS}

\section{Introduction}
%
Consider a flux vector field  \( \vec{q} \) that satisfies a balance equation of the form
\[
 \nabla \cdot \vec{q} = f.
\]
The vector field \( \vec{q} \) is the flux associated with an unknown scalar quantity \( u \). The relationship between \( \vec{q} \) and \( u \) is mediated by a random parameter \( a \), which is subject to abstract structural assumptions. These assumptions are designed to capture a broad range of parametric dependencies while ensuring sufficient smoothness of the solution with respect to the underlying random parameters. Techniques to verify a parametric regularity assumption for varied discretizations of $\vec q$  are a key focus of this paper. 

A key motivating example is the diffusion equation, where the parameter \( a \) corresponds to the inverse of the diffusivity. In this case, the constitutive relation takes the form
\[
 \vec{q} = -a^{-1} \nabla u.
\]
We assume that the parameter field \( a \) is modeled as a Gevrey-regular random field. More precisely, we consider random fields that can be represented as smooth functions of countably infinite sequences of independent random variables, which may follow uniform or Gaussian distributions. This setting includes a wide class of parametric models arising in uncertainty quantification and partial differential equations (PDEs) with stochastic parameters.

Under further assumptions, we establish dimension-independent error bounds for quasi-Monte Carlo (QMC) methods applied to compute expected values of quantities of interest. These quantities are assumed to depend continuously on either the solution \( u \), the flux \( \vec{q} \), or the gradient \( \nabla u \). Our analysis shows that \corr{dimension-independent and faster-than-Monte Carlo} convergence rates of QMC methods can be achieved in this setting, provided the integrands inherit sufficient regularity from the underlying parametric dependence.

The theoretical framework developed in this work applies to a wide variety of balance equations discretized using finite element methods. We show that the required assumptions are satisfied for the diffusion equation when discretized with conforming finite elements, mixed finite element methods, and hybridizable discontinuous Galerkin (HDG) schemes. 

Finally, we present numerical experiments that corroborate our theoretical results. These experiments emphasize the central role of the flux \( \vec{q} \) in QMC-based uncertainty quantification. In particular, we observe that accurate approximation of the flux is critical for achieving the predicted convergence rates, as many quantities of interest depend on the flux either directly or indirectly through its influence on \( u \) and \( \nabla u \).

\subsection{The prominent example of the diffusion equation}
%
The parametric diffusion equation solves for a function $u$ of parameter $\omega$ and spatial variable $\vec x$. Here $\omega$ is a random parameter sample from a probability space \( (\Omega, \mathcal F, \mathbb P) \) and $\vec x$ is an element of the spatial domain \( D \subset \IR^d \), a bounded open set \corr{with typically small spatial} dimension (e.g.,~\( d \in \{1,2,3\} \)).  Let \( L^2(D) \) denotes the space of square-integrable functions on \( D \) and let \( H^1_0(D) \) denote its standard Sobolev subspace of functions having square-integrable first-order weak derivatives and vanishing trace on the spatial boundary $\d D$. The parameter-dependent solution $u \in L^2(\Omega, \mathcal F, \mathbb P; H^1_0(D))$ is such that for almost every $\omega \in \Omega$,
\begin{equation}\label{EQ:base_pde}
 \int_D a^{-1}(\omega) \nabla u(\omega) \cdot \nabla v \dx = \int_D f v \dx \qquad \text{ for all } v \in H^1_0(D),
\end{equation}
where $f \in L^2(D)$, and $a(\omega), a^{-1}(\omega) \in L^\infty(D)$ are positive functions on $D$. In many applications, the coefficient $a$ is used to model spatially varying material properties, with uncertainties described probabilistically (cf., e.g.,~\cite{cdv}). We follow the common approach where  $a^{-1}$ is assumed to be
derived from a random field, i.e.,  we identify 
\[
 a(\vec y) := a(\vec y(\omega)) := a(\omega) \qquad \text{ and } \qquad u(\vec y) := u(\vec y(\omega)) := u(\omega),
\]
where $\vec y = (y_m)_{m\in\mathbb N} \in U \subset \IR^\IN$ is a sequence of independent random variables on the probability space $(\Omega, \mathcal F, \mathbb P)$. Examples of $U$ are given in Section~\ref{SEC:qmc}. \corr{In practice, the infinite-dimensional parameter set $U$ is truncated into an $s$-dimensional space $U_s$, where $s$ is the truncation dimension, not to be confused with the spatial dimension $d$.} In this work, we consider a problem in which the reciprocal of the random field $a$ enters~\eqref{EQ:base_pde}---although this is not a significant change in the lognormal setting. This setting leads naturally to mixed finite element methods that are exceptionally useful if the quantity of interest is related to the flux of the flow field of some physical phenomenon. 

\subsection{Formulations in terms of fluxes}
%
\corr{Virtually all analyses of quasi-Monte Carlo methods for
  uncertainty quantification of partial differential equations with
  random coefficients employ a spatially conforming, primal finite
  element discretization. However, in many cases, it is preferable to
  consider non-conforming or dual finite element methods. For example,
  in groundwater flow simulations and coupled (contaminant) transport
  problems, such schemes are essential for obtaining mass
  conservation. Our analysis provides an abstract framework in which
  primal, dual, and several other methods can be analyzed in the same
  manner if the inverse of the random coefficient is Gevrey regular,
  cf.\ Section \ref{SEC:parametrization}.}

Mixed formulations rewrite \eqref{EQ:base_pde} as two subproblems, namely, (i) the balance equation \( \nabla \cdot \vec q(\vec y) = f \), which is the diffusion equation if (ii) the constitutive equation is modeled according to Fick's law \( \vec q(\vec y) = - a^{-1} (\vec y) \nabla u(\vec y) \). This manuscript investigates a unified framework that allows the analysis of parametric regularity bounds of such continuity equations for several finite element schemes using the parametric diffusion equation \eqref{EQ:base_pde} as a fundamental example.

While many authors have already investigated the (continuous, classic) finite element method for \eqref{EQ:base_pde}, see the references in Section \ref{SEC:parametrization}, there are very few works  exploiting the properties of the dual variable \(\vec q(\vec y)\). Notable examples thereof include \cite{GrahamSU15} that uses multilevel Monte Carlo methods, \cite{GrahamKNSS11} that applies mixed methods in concert with quasi-Monte Carlo (QMC) methods, and \cite{ErnstS14}, which considers the mixed formulation for lognormal random fields, for which $a^{-1}$ is also lognormal. The latter proves analytic dependence on the parameter in this case, yielding a parametric regularity result for the undiscretized flux.  

Numerical schemes that directly discretize the flux unknown producing a discretization \( \vec q_h(\vec y) \) of $\vec q(\vec y)$, often called mixed methods, or  dual methods, are  preferred  when mass conserving numerical approximations are needed. Mass conservation, when the spatial domain is discretized by a finite element mesh,  entails  that for any domain \( S \subset D \) that is a union of mesh elements, the net outward flux through its boundary $\partial S$ is captured with no error, i.e., 
\[
  \int_{\partial S} \vec q_h(\vec y) \cdot \vec n \ds  = \int_{\partial S} \vec q(\vec y) \cdot \vec n \ds.
\]
This relation immediately implies that expectation of the discrete flux field is also conservative for such methods.  Note also that since the last term above equals \( \int_S \nabla \cdot \vec q(\vec y) \dx = \int_S f \dx \), when the source term $f$ is not random, the outflux cannot vary with the random $\vec y$ for conservative methods. For efficiency, one often uses hybridized versions of such methods, e.g.,  the HDG (hybridizable discontinuous Galerkin) method produces conservative fluxes while allowing static condensation for efficiency. To manage the complexity of analysis of such methods for parametric PDEs, we provide a general framework to easily apply  standard results on complicated finite element methods and obtain results for  parametric PDEs. It offers a convenient way to combine state-of-the-art finite elements and high-dimensional integration formulas.

\subsection{Parametrization of randomness}\label{SEC:parametrization}
%
The majority of studies on the application of QMC methods to PDEs with random coefficients have been carried out under the assumption of a fixed parametric representation for the input random field: some commonly used parametric models include the so-called uniform and affine model~\cite{cds10,spodpaper14,dicklegiaschwab,ghs18,SchwabG11,kss12,kssmultilevel,schwab13} and the lognormal model~\cite{gittelson,GrahamKNSSS15,GrahamKNSS11,log3,log4,log5,schwabtodor}. Recently, several studies have gone beyond these models, focusing instead on the class of Gevrey regular input random fields~\cite{ChernovL24,ChernovL24b,Harbrecht24}. This class contains infinitely smooth but generally non-analytic functions with a growth condition imposed on their higher-order partial derivatives. Such representations of input uncertainty can be significantly more general than those previously considered in the literature, and it has been demonstrated in~\cite{ChernovL24,ChernovL24b,Harbrecht24} that it is possible to construct QMC cubature rules that achieve dimension-independent, faster-than-Monte Carlo cubature convergence rates for the quantification of uncertainties in stochastic models. This has also been done under abstract/general assumptions to the considered problem \cite{GuKa24}.

In particular, the works above show that if the input random field is sufficiently smooth with respect to the uncertain variables, then the PDE response will also be smooth. The benefit of using cubature methods such as QMC over Monte Carlo methods is that they can exploit the smoothness of the PDE response, yielding faster convergence rates for the computation of the response statistics. Specific QMC integration methods---e.g., lattice rules~\cite{kuonuyenssurvey,kuonuyenssurvey,dks13,dick2022lattice}---can be trivially implemented in parallel, making them ideal for large-scale computations.

\corr{While there exists a theory for higher-order QMC in the uniform and bounded setting \cite{spodpaper14,hoqmcpaper}, there is currently no complete higher-order QMC theory in the unbounded setting. To this end, we focus on lattice rules, which can be shown to work in both settings. That is, the aim of this work is to develop a cohesive theoretical framework for both, the bounded and unbounded settings rather than extending the higher-order QMC theory.
}

\subsection{Notation}
%
Infinite multi-indices are sequences of non-negative integers of the form $\bd \nu = (\nu_1, \nu_2, \dots)$.
Let $\mathscr F$ denote the set of infinite multi-indices $\bd \nu$ of finite length $|\bsnu| \corr{:=\sum_{j=1}^\infty \nu_j} <\infty$, i.e., $\mathscr F$ is the set of all finitely supported infinite multi-indices. Also let
\begin{equation} \label{eq:F1}
 \mathscr F_1 = \{ \bd \nu \in \mathscr F: \; \nu_j \in \{0,1\} \text{ for all } j \ge 1\}.
\end{equation}
Let $\bsnu,\boldsymbol m\in\mathscr F$ and let $\boldsymbol x=(x_j)_{j\geq 1}$ be a sequence of real numbers. We shall employ the shorthand notations
\begin{gather*}
 \boldsymbol m\leq \bsnu\quad\text{if and only if}\quad m_j\leq \nu_j\quad\text{for all}~j\geq 1,\\
 \binom{\bsnu}{\boldsymbol m}:=\prod_{j\geq 1}\binom{\nu_j}{m_j}, \qquad \partial_{\boldsymbol x}^{\bsnu}:=\prod_{j\geq 1}\frac{\partial^{\nu_j}}{\partial x_j^{\nu_j}}, \qquad
 \bsnu!:=\prod_{j\geq 1}\nu_j!, \qquad \boldsymbol x^{\bsnu}:=\prod_{j\geq 1}x_j^{\nu_j},
\end{gather*}
where we use the convention $0^0:=1$. We use $\{1:s\}$ to abbreviate the set $\{1,\ldots,s\}$ for any positive integer $s$.  For any subset $\setu$ of $\{1:s\}$, we use $|\setu|$ to denote the cardinality of the set $\setu$. In addition to the multi-index derivative notation $\d_{\bd x}^{\bd\nu}$ above, we associate a derivative operator to the set $\setu$ given by
\begin{equation}\label{eq:d-setu}
 \frac{ \d^{|\setu|}}{ \d \bd x_{\setu} } = \prod_{j \in \setu } \frac{\d }{\d x_j} ,
\end{equation}
where the product of the first-order derivative operators on the right indicates their composition. Note that this is a mixed derivative that does not differentiate more than once in any given direction $x_j$ with  $j \in \setu$.

\subsection{Main assumptions}\label{SEC:assumptions}
%
In what follows, we analyze the quasi-Monte Carlo (QMC) method applied to a general balance law partial differential equation (PDE). To this end, we formulate abstract assumptions regarding the discretized PDE, the quantity of interest, and the random coefficient in Section \ref{SEC:setting}. Specifically, we require the finite element approximation to satisfy the following: the energy norm of the numerical flux must dominate its \( L^2 \) norm; the method must be stable; and a recursive bound must hold for the parametric derivatives.

Regarding the quantity of interest, we assume that it, along with its parametric derivatives, is controlled by the flux and its corresponding derivatives. As for the random coefficient, we assume it belongs to the Gevrey class of order \( \sigma \), and that the ratio between its supremum and infimum can be appropriately bounded.

On the one hand, these assumptions enable us to carry out a complete QMC error analysis. On the other hand, they are also relatively easy to verify, as demonstrated in Section \ref{SEC:finite_elements}.

\subsection{Structure of this document}
Section \ref{SEC:qmc} introduces the QMC method as one of the prominent methods whose analysis can be conducted by verifying certain assumptions, while Section \ref{SEC:para_reg} derives the parametric regularity bounds that are needed to complete the error analysis of  QMC methods. Section \ref{SEC:finite_elements} introduces several finite element methods that fall into our framework, while Section \ref{SEC:numerics} confirms our analytical findings numerically. A short conclusions section rounds up this exposition.

\section{Quasi-Monte Carlo method}\label{SEC:qmc}
%
The quasi-Monte Carlo (QMC) method recalled in this section  can approximate integrals of functions supported in a high-dimensional cube.  Let $F\colon \mathcal I^s \to \IR$ be a continuous function, where $s \in \mathbb N$ denotes some (high) dimension and $\mathcal I$ denotes an interval (not necessarily bounded) of $\IR$. Given a \corr{strictly positive} probability density function $\varphi: \mathcal I \to (0, \infty)$, we are interested in approximating the integral
\[
 I_s^{\varphi}(F) := \int_{\mathcal I^s}F(\bsy)\prod_{j=1}^s \varphi(y_j) \,{\rm d}\bsy,
\]
i.e., the integral of $F$ with respect to the \corr{product measure on $\mathcal I^s$ induced by the density $\varphi$ on $\cl I$, provided that \( F \) is integrable.}

We first transform the function $F$ to the fixed reference domain $(0,1)^s$ where the cubature is implemented. Let $\hat\phi: \mathcal I \to (0,1)$ be a univariate function satisfying $\varphi(y) = \hat \phi'(y)$. Clearly, since $\varphi > 0$ is a probability density, such an antiderivative of $\varphi$ with values in $(0,1)$ can be found and it is monotonically increasing.  Define the multivariate function $\Phi\colon \mathcal I^s \to (0,1)^s$ component-wise by \( (\Phi(\vec y))_j = \hat \phi (y_j) \), for all $\vec y \in \cl I^s$. Since $\hat \phi$ is one-to-one, the map $\Phi$ is invertible and has a positive Jacobian determinant \( \det D \Phi(\bsy) = \prod_{j=1}^s \varphi(y_j) \). Using the change of variable formula, with $\bd y = \Phi^{-1}(\bd t)$ and ${\rm d} \bd t = (\det D \Phi(\bd y) ) \;{\rm d}\bd y$,
\begin{equation}\label{EQ:eval_int}
 I_s^{\varphi}(F) = \int_{\mathcal I^s}F(\bsy)\prod_{j=1}^s \varphi(y_j)\,{\rm d}\bsy=\int_{(0,1)^s}F(\Phi^{-1}(\boldsymbol t))\,{\rm d}\boldsymbol t.
\end{equation}
Hence it suffices to develop a cubature on $(0,1)^s$.

The QMC method we consider in this paper approximates \(I_s^{\varphi}(F)\) using $R$ random shifts $\bd \Delta_r \in [0, 1]^s$ for $r=1, \ldots, R$, each a sample from a uniform distribution in $[0,1]^s$. Using only the $r$th random shift, define
\[
  Q_r^\varphi (F) = \frac 1 n \sum_{k=1}^n F(\Phi^{-1}(\{\boldsymbol t_k+\boldsymbol\Delta_r\})),
\]
where $\{\bd v\}$ denotes the component-wise fractional part of $\bd v$, which using the floor function~$\lfloor{\cdot}\rfloor$ can be expressed as $\{\bd v \} = (v_1 - \lfloor{v_1}\rfloor, v_2 - \lfloor{v_2}\rfloor, \dots)$, and $\bd t_k$ are the lattice points defined by \( \boldsymbol t_k=\{\tfrac{k\boldsymbol z}{n}\} \) for \(k\in\{1,\ldots,n\}\) and some $\boldsymbol z\in\mathbb N^s$, an efficiently computable generating vector.  Using $R$ such shifts, the \emph{randomly shifted rank-1 lattice QMC rule} is given by
\[
 \Qmc(F) = \frac{1}{R}\sum_{r=1}^R Q_r^\varphi(F).
\]
When $F$ is understood from the context, we will omit it from the notation and just write $I_s^\varphi, Q_r^\varphi$ and $\Qmc$ instead of $I_s^\varphi(F), Q_r^\varphi(F)$ and $\Qmc(F)$, respectively.

The random shifts $\boldsymbol\Delta_1, \ldots, \boldsymbol\Delta_R$ are independently and identically distributed  (i.i.d.) random variables in $[0,1]^s.$ The expected value of the QMC quadrature with respect to
these random shifts is the $R$-fold integral
\[
 \EE   [\Qmc] = \int_{\underbrace{\text{\scriptsize$ (0,1)^s \times \dots \times (0,1)^s$}}_{R \text{ times}}} \Qmc \; \textup d (\boldsymbol\Delta_1 \otimes \dots \otimes \boldsymbol\Delta_R).
\]
The QMC approximation is unbiased, i.e., its mean is the exact integral,
\begin{equation} \label{eq:qmc-expectation}
 \EE  [\Qmc] = \EE[ Q_r^\vphi] = I_s^{\varphi},
\end{equation}
and its variance decreases like $1/R$, i.e., 
\begin{equation}\label{eq:qmc-variance}
 \EE  \big[\left(I_s^\varphi -\Qmc \right)^2\big] = \frac{1}{R}\EE   \big[\big(I_s^\varphi - Q_r^\varphi \big)^2\big] 
\end{equation}
for any $r=1, \ldots, R.$ Proofs of~\eqref{eq:qmc-expectation} and \eqref{eq:qmc-variance} are standard, but are included in Appendix~\ref{sec:more-details} for completeness. Note that in~\eqref{eq:qmc-expectation} and \eqref{eq:qmc-variance}, the first $\EE$ is over all $\bd\Delta_1,\ldots \bd\Delta_R$ while the second $\EE$ is just over $\bd\Delta_r$ for any fixed~$r$.

The quadrature error, as represented by the square root of the left hand side of~\eqref{eq:qmc-variance}, thus decreases like $R^{-1/2}$. In what follows, we recall results showing how it also decreases with~$n$, provided $F$ satisfies a regularity condition. We do so for two cases, one for the uniform distribution on the bounded domain $\mathcal I^s = [-\tfrac12,\tfrac12]^s$, and another for the multivariate normal distribution on the unbounded domain $\mathcal I^s = \IR^s$.

\subsection{Bounded domain \( [-\tfrac12,\tfrac12]^s \)}\label{ssec:bounded-domain}
%
Consider the case of \corr{$\mathcal I^s = [-\tfrac12, \tfrac12]^s$}.  Since the Lebesgue measure of \( [-\tfrac12,\tfrac12]^s \) is one, its indicator function
\begin{equation}
\varphi\colon [-\tfrac12,\tfrac12] \to \mathbb R,\qquad\varphi(y)=\mathbf 1_{[-1/2,1/2]}(y),\label{eq:phiunif}
\end{equation}
is a probability density function on $\cl I^s$.  Then $\hat \phi^{-1}(t)=t-\tfrac12$ for $t\in [0,1]$. Note that $\vphi$ in this case corresponds to the probability density function of the uniform distribution on $[-\tfrac12,\tfrac12]$ denoted by $\mathcal U([-\tfrac12,\tfrac12])$.

Suppose that we are given a collection $\gamma$ of some positive numbers $\gamma_{\mathfrak{u}}$ indexed by  subsets  $\mathfrak u$ of $\{1:s\}$. Define an unanchored weighted Sobolev norm
\begin{equation}\label{EQ:bounded_norm} \|F\|_{s,\boldsymbol\gamma}=\bigg(\sum_{\setu\subseteq\{1:s\}}\frac{1}{\gamma_{\setu}}\int_{\left[-\tfrac12,\tfrac12\right]^{|\setu|}}
  \bigg(\int_{\left[-\tfrac12,\tfrac12\right]^{s-|\setu|}} \; \frac{\partial^{|\setu|}}{\partial \vec y_\setu} F(\bsy)\,{\rm d}\bsy_{-\setu}\bigg)^2\,{\rm d}\bsy_{\mathfrak u}\bigg)^{\frac12},
\end{equation}
where the sum runs over all subsets $\setu$ of $\{1:s\}$, the inner integral (${\rm d}\bsy_{-\setu}$) is over all $y_j$ with $j \in \{1:s\} \setminus \setu$ (called the inactive variables), while the outer integral (${\rm d}\bsy_{\mathfrak u}$) is over all $y_j$ with $j \in \setu$ (called the active variables). It is known~\cite{ckn06,dks13,cn06} that a generating vector can be obtained by a component-by-component (CBC) algorithm resulting in a QMC rule with the following error bound.

\begin{theorem}[\corr{{\cite[Theorem~5.1]{kuonuyenssurvey}}}]\label{THM:bounded}
 Let $\varphi$ be defined by~\eqref{eq:phiunif}, $s$ be an arbitrary dimension, and let the weighted norm of $F$ in \eqref{EQ:bounded_norm} be bounded for some given weight collection $\gamma$.  A randomly shifted QMC rank-1 lattice rule with $n=2^m$ in $s$ dimensions can be constructed by a CBC algorithm with $R$ independent random shifts such that, for $\lambda\in(1/2,1]$, there holds
 \[
  \sqrt{\EE\Big[ \left(I_s^\varphi (F)-\Qmc(F)\right)^2\Big]} \leq\; \frac{C_{s,\gamma,\lambda}}{\sqrt{R\; n^{1/\lambda}}} \;\|F\|_{s,\gamma},
 \]
 where the constant $C_{s,\gamma,\lambda}$ can be given using the Riemann zeta function $\zeta(x)=\sum_{k=1}^\infty k^{-x}$, $x>1$, by
 \begin{equation}\label{eq:Csgammalambda}
  C_{s,\gamma,\lambda} = \bigg({2}\sum_{\varnothing\neq\mathfrak u\subseteq\{1:s\}}\gamma_{\mathfrak u}^{\lambda}\bigg(\frac{2\zeta(2\lambda)}{(2\pi^2)^\lambda}\bigg)^{|\setu|}\bigg)^{\frac{1}{2\lambda}},
 \end{equation}
 and $\EE $ denotes the expected value with respect to the uniformly distributed random shifts~$\bd\Delta$.
\end{theorem}

\subsection{Unbounded domain \( \IR^s \)}\label{ssec:unbounded-domain}
%
The second case of interest is $\cl I^s = \IR^s$ with 
\begin{equation}
 \varphi\colon \mathbb R\to \mathbb R,\qquad \varphi(y)=\frac{1}{\sqrt{2\pi}}{\rm e}^{-\frac12y^2}.\label{eq:normphi}
\end{equation}
Then  $\hat\phi^{-1}(t)=\frac12\big(1+{\rm erf}\big(\frac{t}{\sqrt 2}\big)\big)$ for $t\in(0,1)$  where  \( {\rm erf}(\cdot) \)  denotes the Gauss error function. This  $\varphi$ corresponds to the probability density function of the multivariate normal distribution $\mathcal N(\mathbf 0, \mathbf{I}_{s\times s})$ of vanishing mean and variance $\mathbf{I}_{s\times s}$, the $s\times s$ identity matrix. \corr{Notably, using this specific function is a choice that we make for our manuscript. This choice, however, is not canonical and is possibly not ideal as indicated in \cite{KazashiSG25,KritzerPPW20}.}

To describe the regularity needed from $F$ in this case, suppose we are given a $\gamma$ as before (see~Section~\ref{ssec:bounded-domain}). In addition, suppose we are given a sequence of positive numbers $\alpha = (\alpha_j)_{\corr{j \in \bb N}}.$ Letting
\[
 \varpi_j(x)=\exp(-\alpha_j|x|),\qquad\alpha_j>0,~j\in\mathbb N, \quad x \in \bb R,
\]
define a weighted Sobolev norm $\| \cdot \|_{s, \gamma, \alpha}$ by
\begin{equation}\label{eq:gaussnorm}
 \|F\|_{s,\gamma,\alpha}^2 =\sum_{\setu\subseteq\{1:s\}}\frac{1}{\gamma_{\setu}}\int_{\mathbb R^{|\setu|}}\bigg(\int_{\mathbb R^{s-|\setu|}}\frac{\partial^{|\setu|}}{\partial \vec y_\setu} F(\bsy)\prod_{j\in\{1:s\}\setminus\setu}\varphi(y_j)\,{\rm d}\bsy_{-\setu}\bigg)^2\prod_{j\in\setu}\varpi_j^2(y_j)\,{\rm d}\bsy_{\setu}.
\end{equation}
Then it is known that a generating vector can be obtained by a CBC algorithm~\cite{nicholskuo} satisfying the following rigorous error bound.

\begin{theorem}[cf.~{\corr{\cite[Theorem~7]{nicholskuo}}}]\label{thm:lognormalqmcerror}
 Let $\varphi$ be defined by~\eqref{eq:normphi}, $s$ be an arbitrary dimension, and let $\| F \|_{s, \gamma, \alpha}$ be bounded for some $\gamma$ and $\alpha$.  A randomly shifted QMC rank-1 lattice rule with $n=2^m$ in $s$ dimensions can be constructed by a CBC algorithm with $R$ independent random shifts such that, for any  $\lambda\in(1/2,1]$, we have 
 \[
  \sqrt{\EE\Big[ \left(I_s^\varphi (F)-\Qmc(F)\right)^2\Big]} \le \frac{C_{s, \gamma, \lambda, \alpha}}
  {\sqrt{R \; n^{1/\lambda}}} \|F\|_{s,\gamma,\alpha},
 \]
 where $\EE $ denotes the expected value with respect to the uniformly distributed random shifts $\bd\Delta_r$ in $[0,1]^s$.  The constant $C_{s, \gamma, \lambda, \alpha}$ can be explicitly given using the Riemann zeta function $\zeta$ and $\eta = (2\lambda -1) (4\lambda)^{-1}$ by
 \begin{equation}\label{EQ:cgla_log}
  C_{s, \gamma, \lambda, \alpha} = \bigg(2\sum_{\varnothing\neq\setu\subseteq\{1:s\}}\gamma_{\setu}^\lambda \prod_{j\in\setu}\varrho_j(\lambda)\bigg)^{\frac{1}{2\lambda}}, \qquad  \varrho_j(\lambda):= 2\bigg(\frac{\sqrt{2\pi}\exp(\alpha_j^2/\eta)} {\pi^{2-2\eta}(1-\eta)\eta}\bigg)^\lambda \zeta(\lambda+\tfrac12).
 \end{equation}
\end{theorem}

We now proceed to show that under some reasonably general conditions, we can exploit Theorems~\ref{THM:bounded} and~\ref{thm:lognormalqmcerror} to obtain error estimates for the QMC approximation of some quantity of interest that depends on the finite element flux~$\bd q_h$.

\section{Parametric regularity estimates}\label{SEC:para_reg}
%
In this section, we first present sufficient conditions under which the parametric regularity estimates—required for the QMC error analysis—can be established. These conditions pertain to the discretized PDE, the quantity of interest, and the structure of the random coefficient, which is assumed to exhibit Gevrey regularity.

Subsequently, in Section \ref{SEC:bounded_qmc}, we derive non-recursive parametric regularity estimates and the corresponding QMC error bounds for a bounded parameter domain. Finally, we extend the analysis to the case of an unbounded domain.

\subsection{The setting}\label{SEC:setting}
%
We want to approximate functionals of a flux $\bd q$ that depend on some random coefficient $a(\bd x, \omega)$ where $\bd x$ is in $D$, a Lipschitz polyhedron in $\bb R^d$, and $\omega$ is a sample point in a probability space $(\Omega, \mathcal F, \mathbb P)$. Following prior authors (see e.g., \cite{kuonuyenssurvey}) we assume that the random coefficient $a(\bd x, \og) $ has been parameterized by a vector $\bd y = (y_1(\omega), y_2(\omega), \cdots )$ where every component $y_j$ lies in some (bounded or unbounded) real interval $\cl{I}$, i.e., $\bd y$ is in \( U = \cl{I}^{\IN} \). Endowing $\cl{I}$ with a probability measure generated by a probability density function $\varphi$ on $\cl{I}$, the set $U$ inherits a product measure which we denote by $\mu({\rm d} \bd y)$.  Writing $a(\bd y)$ for $a(\cdot, \bd y)$, we assume that $a(\bd y) >0$ a.e.\ in $D$ \corr{(i.e., for each $\vec y$ we have $a(\vec y)=a(\vec x,\vec y)>0$ for Lebesgue-a.e.~$\vec x\in D$)} and that \( a \) is Gevrey-\( \sigma \) regular, a property that will be rigorously formulated in \Cref{IT:gev}.

The exact flux $\vec q$ is in $ L^2(\Omega, \mathcal F, \mathbb P; L^2(D))$ and has the property that $\vec q(\omega) \in H(\dive, D)$ for almost every $\omega \in \Omega$.  While $\bd q$ solves the balance law $\dive \bd q =f$, its discrete approximation $\bd q_h$ in some (conforming or nonconforming) finite element space $\mathcal W_h$ with norm $\| \cdot \|_{\mathcal W_h}$, solves a discretization of the law (such as those considered in \Cref{SEC:finite_elements}). Here $h$ denotes some meshsize parameter related to the discretization. The dependence of the exact and discrete fluxes on the random coefficient is henceforth indicated by $\bd q(\bd y)$ and $\bd q_h(\bd y)$, which are both functions of the spatial variable $\bd x$ in $D$, and are in $H(\dive, D)$ and $\mathcal W_h$, respectively for $\bd y$ in $U$.

Suppose a quantity of interest is represented by a real-valued (not necessarily linear) functional $J$ of the flux, $J : H(\dive, D) + \mathcal W_h \to \bb R$, also called the goal functional. A typical example is $J(\bd q) = \int_D a^{-1} \bd q \cdot \bd q \; \dx $ (see Example~\ref{eg:quadraticJ} below). We are interested in approximating the expected value of $J(\bd q(\bd y))$ with respect to the random variables parameterized by $\bd y$, namely
\[
  \EE[ J(\bd q) ] = \int_U J(\bd q(\bd y)) \; \mu(\dy)
\]
where the integral is computed using the above-mentioned product measure $\mu(\dy)$ generated by $\varphi$.  To approximate this using the QMC rule of Section~\ref{SEC:qmc}, we first truncate the infinite sequence $\bd y$ to its first $s$ components $\bd y_s = (y_1, \dots, y_s, 0,0, \dots)\corr{\in U}$ and set $\bd q_s = \bd q(\bd y_s)$. A computable finite element flux also employs the same $s$ components, $\bd q_h \equiv \bd q_h(\bd y_s)$. \corr{Where convenient, we occasionally abuse notation and identify $\vec q_h(y_1,\ldots,y_s)\equiv \vec q_h(y_1,\ldots,y_s,0,0,\ldots)$ for $\vec y\in U_s$.} Having truncated the parameter space dimension to the finite dimension~$s$, we may now use the QMC rule of Section~\ref{SEC:qmc} to approximate the integral. Summarizing these standard steps, the final approximation error is split into dimension truncation error, finite element error, and cubature error, namely
\begin{align*}
 \Big|\mathbb E[J(\vec q)]-\Qmc(J(\vec q_h)) \Big| &\leq \Big|\mathbb E[J(\vec q)] - \mathbb E[J(\vec q_s)\Big| +
    \Big|\mathbb E[J(\vec q_s)] - \mathbb E[J(\vec q_h)] \Big| \\
 & + \Big| \mathbb E[J(\vec q_h)] - \Qmc(J(\vec q_h)) \Big|,
\end{align*}
where all expectations are computed with respect to~$\bd y$.  The former two errors are well-studied in prior literature. The focus of this study is limited to the last term.

We now state a few conditions under which the last term representing the QMC error after dimension truncation can be bounded. Henceforth, \corr{unless otherwise stated}, we assume that $\bd y$ is already dimension-truncated, i.e., $\bd y \equiv \bd y_s$ from now on. Note however that all bounds we derive will be $s$-independent, so the dimension $s$ can be made arbitrarily close to infinity without any deterioration in the constants.  In our conditions below, all constants denoted by $C$ with some subscript are assumed to be independent of~$s$ and $h$. The first two conditions below are usually immediate in methods generating flux approximations. \corr{It refers to a general norm \( \| \cdot \|_{\mathcal W_h} \) that can be chosen freely such that all assumptions hold in that norm.}
\begin{assumption}\label{IT:norm} \textbf{\( \| \cdot \|_{\mathcal W_h} \) is stronger than the energy norm.} There is a constant \( C_E > 0 \) such that  for all $\vec y \in \corr{U_s}$ and all subsets $\setu$ of $\{1:s\}$
 \[
  \left\| \sqrt{a(\vec y)} \; \frac{\partial^{|\setu|}}{\partial \vec y_\setu} \vec q_h(\vec y) \right\|_{L^2(D)} \le C_E \left\| \frac{\partial^{|\setu|}}{\partial \vec y_\setu} \vec q_h(\vec y) \right\|_{\mathcal W_h}.
 \]
\end{assumption}
\begin{assumption}\label{IT:stab} \textbf{Stability of the numerical method.} There is a constant \( C_S > 0 \) with
 \[
  \| \vec q_h(\vec y) \|_{\mathcal W_h} \le C_S \| a(\vec y)\|^{1/2}_{L^\infty(D)} \| f \|_{L^2(D)} \qquad \text{ for all } \vec y \in \corr{U_s}.
 \]
\end{assumption}
A key parametric regularity condition that is needed in the  QMC analysis is the following estimate. Such results follow by differentiating the discretized equations with respect to the parameter and leveraging the scheme's stability---see e.g., a simple verification in~Lemma~\ref{LEM:mm_parametric}.
\begin{assumption}\label{IT:recur}
 \textbf{Recursive bound for the parametric derivative.} There is \( C_R > 0 \) such that for any \( \vec \nu \in \mathscr F_1 \) (see \eqref{eq:F1}) and $\vec y \in \corr{U_s}$, we have
 \[
  \| \der \vec q_h(\vec y) \|_{\mathcal W_h} \le C_R \sum_{\vec 0 \neq \vec m \le \vec \nu} \begin{pmatrix} \vec \nu \\ \vec m \end{pmatrix} \left\| \frac{\derm a(\vec y)}{a(\vec y)} \right\|_{L^\infty(D)} \| \dernum \vec q_h(\bsy) \|_{\mathcal W_h}.
 \]
\end{assumption}
The next assumption requires the quantity of interest $J(\vec q)$ to be bounded by some power of the $L^2(D)$-norms of the parametric derivatives of flux. We show right away that it holds for continuous linear functionals and the quadratic energy functional.
\begin{assumption}\label{IT:qoi}
 \textbf{Boundedness of the quantity of interest.}  There is \( C_J > 0 \)  and \( r >  0 \) such that for any \( \setu \subset \{1:s\} \) and  every \( \vec y \in \corr{U_s} \), 
 \[
  \left| \frac{\partial^{|\setu|}}{\partial \vec y_\setu} J(\vec q_h(\bsy)) \right| \le C_J^r \sum_{\setv \in 2^\setu}\left\| \frac{\partial^{|\mathfrak  v|}}{\partial \vec y_{\mathfrak  v}} \vec q_h(\vec y) \right\|^r_{L^2(D)}.
 \]
 \corr{Here and in what follows, \( \setv \in 2^\setu \) says that \( \setv \) is an element of the power set of \( \setu \), which is equivalent to \( \setv \subset \setu \).}
\end{assumption}
\begin{example}
 If $J$ is a continuous linear functional then \Cref{IT:qoi} holds with $r=1$, since we have \( | \tfrac\partial{\partial \vec q_h} J(\vec q_h) | \le C_J \) for all \( \vec q_h \) some constant \( C_J \), and the chain rule dictates that
 \[
  \left| \frac{\d^{|\setu|}}{\d \bd y_\setu} J(\vec q_h(\vec y)) \right| = \left|J\bigg(\frac{\d^{|\setu|}}{\d \bd y_\setu}\vec q_h(\vec y)\bigg) \right|\leq C_J\bigg\|\frac{\d^{|\setu|}}{\d \bd y_\setu}\vec q_h(\vec y)\bigg\|_{L^2(D)},%
 \]
 which is an even stronger estimate than \Cref{IT:qoi} with \( r = 1 \).
\end{example}
\begin{example}\label{eg:quadraticJ}
 If  $J$ is the  quadratic functional
 \begin{equation}\label{eq:Jquadratic-qq}
  J(\boldsymbol q_h(\bsy)) =\int_D \boldsymbol q_h(\bsy) \cdot \boldsymbol q_h(\bsy)\,{\rm d}\bsx,
 \end{equation}
 then \Cref{IT:qoi} holds. This can be seen using the Leibniz rule for the product of two functions $f(\bd y)$ and $g(\bd y)$, whereby 
 \begin{equation}\label{eq:Leibniz}
 \frac{\d^{|\setu|}}{\d \bd y_\setu}(fg) = \sum_{\setv \in 2^\setu} \,\frac{\d^{|\setv|}f}{\d \bd y_\setv} \, \frac{\d^{|\setu \setminus \setv|}g}{\d \bd y_{\setu \setminus \setv}} 
 \end{equation}
 using the notation in \eqref{eq:d-setu}. Clearly~\eqref{eq:Leibniz} is just the standard Leibniz rule when $\setu$ is a singleton $\{ i\}$. One may use induction on $|\setu|$ to establish~\eqref{eq:Leibniz}. Selecting any element, say $i$ in $\setu$ and putting $\setw = \setu \setminus \{ i\}$,
 \[
  \frac{\d^{|\setu|} }{\d \bd y_\setu} (fg) = \frac{\d}{\d y_i } \left( \frac{\d^{|\setw|}}{\d \bd y_\setw  } fg \right),
 \]
 we apply induction hypothesis to the $\setw$-derivative, followed by the standard Leibniz rule. The resulting terms are grouped exactly as in the decomposition of $2^\setu$ into sets of the form $\emptyset \cup \setv$ and $\{i\} \cup \setv$ for all $\setv$ in $2^\setw$, thus proving equation~\eqref{eq:Leibniz}. Applying it to \eqref{eq:Jquadratic-qq},
 \begin{align} \notag
  \frac{\partial^{|\setu|}}{\partial\bsy_{\setu}}J(\boldsymbol q_h(\bsy)) &=\sum_{\mathfrak v\subseteq\mathfrak u} \int_D\frac{\partial^{|\mathfrak v|} \boldsymbol q_h}{\partial \bsy_{\mathfrak v}} \cdot \frac{\partial^{|\mathfrak u\setminus \mathfrak v|} \bd q_h}{\partial \bsy_{\mathfrak u\setminus\mathfrak v}}\,{\rm d}\bsx \leq \sum_{\mathfrak v\subseteq\mathfrak u} \bigg\|\frac{\partial^{|\mathfrak v|}\bd q_h}{\partial\bsy_{\mathfrak v}}\bigg\|_{L^2(D)} \bigg\|\frac{\partial^{|\mathfrak u\setminus \mathfrak v|}\bd q_h}{\partial\bsy_{\setu \setminus \setv}}\bigg\|_{L^2(D)}\notag \\
  &\leq \sum_{\mathfrak v\subseteq\mathfrak u}\bigg\|\frac{\partial^{|\mathfrak v|}}{\partial\bsy_{\mathfrak v}}\boldsymbol q_h(\bsy)\bigg\|_{L^2(D)}^2,\label{eq:quadraticfun}
 \end{align}
 where we used the Cauchy--Schwarz inequality twice. This verifies \Cref{IT:qoi} with $r=2$.
\end{example}

The next important assumption is on the random parameter. This assumption is motivated by \cite{ChernovL24,ChernovL24b}. By Pringsheim's theorem (cf., e.g.,~\cite[p.~169]{pringsheim}), the case $\sigma=1$ in \eqref{eq:CGevrey} below   corresponds to the class of holomorphic functions; the  cases of $\sigma\in[0,1)$ may therefore be viewed in some abstract way as  slightly smoother than holomorphic functions. Since $\sigma=1$ already recovers the \corr{theoretically expected} rate, estimating $1 \le (|\bsnu|!)^{\sigma}\le |\bsnu|!$ for $0\leq \sigma<1$ is effectively lossless. Thus, we avoid including the case $\sigma\in[0,1)$ in the analysis since some later inequalities exploit the fact that $x\mapsto x^\sigma$, $\sigma\ge1$, is convex for $x\geq0$. The assumed cases of $\sigma\ge 1$ covers holomorphic and less regular parametric dependence.

\corr{The Gevrey class consists of functions with a growth condition placed on their higher-order partial derivatives. However, in the unbounded setting, it is typically desirable to let the constant factor of the parametric regularity bound on $\|\partial_{\bsy}^{\bsnu}a(\bsy)\|_{L^\infty(D)}$ depend on the parametric variable $\bsy$ in order to be able to cover, e.g., lognormal random fields. To this end, we follow~\cite{GuKa24} and place the Gevrey assumption on the ratio $\|\partial_{\bsy}^{\bsnu}a(\bsy)/a(\bsy)\|_{L^\infty(D)}$.}
\begin{assumption}\label{IT:gev}
 \textbf{Coefficient of class Gevrey-$\sigma$.} There are constants \(C_G > 0\), \(\sigma \ge 1\), \(p \in (0,1)\), and a sequence \(\vec b = (b_1, b_2, \dots) \in \IR^\IN \) such that for all \( \vec \nu \in \{0,1\}^{\IN} \) and all $\vec y \in U$,
 \begin{equation}\label{eq:CGevrey}
  \left\| \frac{\partial_{\bsy}^{\bsnu} a(\bsy)}{a(\bsy)} \right\|_{L^\infty(D)}\le C_G (|\bsnu|!)^\sigma \vec b^{\bsnu} \qquad \text{ and } \qquad \sum_{j \ge 1} b_j^p < \infty.
 \end{equation}
\end{assumption}
\corr{\begin{remark}Note that Assumption~\eqref{eq:CGevrey} is stated in the non-dimensionally truncated parametric domain $U$.\end{remark}}
\corr{\begin{example}\label{ex:unifex}Assumption~\ref{IT:gev} is satisfied by the classic uniform and affine expansion
$$
a(\bsx,\bsy)=a_0(\bsx)+\sum_{j\in\mathbb N}y_j\psi_j(\bsx),\quad \bsy\in [-\tfrac12,\tfrac12]^{\mathbb N},
$$
where $a_0\in L^\infty(D)$, $\psi_j\in L^\infty(D)$, and $(\|\psi_j\|_{L^\infty(D)})_{j\in\mathbb N}\in\ell^1(\mathbb N)$. By exploiting the fact that
$$
\partial_{\bsy}^{\bsnu}a(\bsx,\bsy)=\begin{cases}a(\bsx,\bsy)&\text{if}~\bsnu=\mathbf 0\\
\psi_j(\bsx)&\text{if}~\bsnu=\boldsymbol e_j\\
0&\text{otherwise}\end{cases}
$$
there especially holds
$$
|\partial_{\boldsymbol y}^{\boldsymbol\nu}a(\boldsymbol x,\boldsymbol y)|\leq \prod_{j\in{\rm supp}(\boldsymbol\nu)}\|\psi_j\|_{L^\infty(D)}.
$$
Dividing through by $a(\boldsymbol x,\boldsymbol y)$ yields
$$
\bigg|\frac{\partial_{\boldsymbol y}^{\boldsymbol\nu}a(\boldsymbol x,\boldsymbol y)}{a(\boldsymbol x,\boldsymbol y)}\bigg|\leq \prod_{j\in{\rm supp}(\boldsymbol\nu)}b_j,\quad \text{where}~b_j=\frac{\|\psi_j\|_{L^\infty(D)}}{a_{\min}}.
$$
Since the higher-order partial derivatives of $a$ vanish for $\boldsymbol\nu\not\leq \mathbf 1$, we can easily extend the bound for higher-order derivatives as
$$
\bigg\|\frac{\partial_{\boldsymbol y}^{\boldsymbol\nu}a(\cdot,\boldsymbol y)}{a(\cdot,\boldsymbol y)}\bigg\|_{L^\infty(D)}\leq (|\boldsymbol\nu|!)^{\sigma}\boldsymbol b^{\boldsymbol\nu},
$$
where $\sigma\geq 1$ is arbitrary. 
\end{example}}
\begin{example}
 The simplest example of such a coefficient in the unbounded case is
 \[
  a(\vec x, \vec y) = \exp\left( \sum_{j \in \IN} \corr{\xi(y_j)\psi_j(\bsx)} \right),
 \]
 \corr{where $\psi_j\in L^\infty(D)$ are as in Example~\ref{ex:unifex}.} 
 If we choose $\xi$ as the identity, we obtain the (classical, analytic) lognormal case, and for the (non-analytic) lognormal Gevrey-\( \sigma \) case, we choose
 \[
  \xi(t)= \operatorname{sign}(t)\exp(-t^{-\omega}) \qquad \text{ with } \qquad \omega =\frac{1}{\sigma - 1} \iff \sigma=1+\frac{1}{\omega}.
 \]
 That is, the (more-general) Gevrey parameters can often be interpreted as transformations of (classical) analytic parameter fields.
\end{example}

To complete this subsection, we point out a consequence of Assumptions \ref{IT:stab}, \ref{IT:recur} and~\ref{IT:gev} (namely Lemma~\ref{lem:norm-bound-y-derivative} below) that we shall use in our analysis. We also recall two other results we need whose proofs can be found elsewhere.

\begin{lemma}[{see~\cite[Lemma~3.1]{GuKa24}}]\label{LEM:guth_k}
 Let $(\Upsilon_{\bsnu})_{\bsnu\in\mathscr F}$ be a sequence of nonnegative real numbers satisfying
 \begin{equation*}
  \Upsilon_{\mathbf 0}\leq K_0\quad\text{and}\quad \Upsilon_{\bsnu}\leq K_1\sum_{\substack{\boldsymbol m\leq \bsnu\\ \boldsymbol m\neq \mathbf 0}}\binom{\bsnu}{\boldsymbol m}(|\boldsymbol m|!)^\sigma \boldsymbol b^{\boldsymbol m}\Upsilon_{\bsnu-\boldsymbol m}\quad\text{for all}~\bsnu\in\mathscr F,
 \end{equation*}
 where $K_0,K_1>0$ and $\sigma\geq 1$. Then there holds
 \begin{equation*}
  \Upsilon_{\bsnu}\leq K_0(K_1+1)^{|\bsnu|}(|\bsnu|!)^{\sigma}\boldsymbol b^{\bsnu}\quad\text{for all}~\bsnu\in\mathscr F.
 \end{equation*}
\end{lemma}
\begin{lemma}\label{lem:norm-bound-y-derivative}
 \Cref{IT:stab}, \Cref{IT:recur} and~\Cref{IT:gev} imply
 \begin{equation}\label{EQ:de-recur_uni}
  \| \der \vec q_h(\vec y) \|_{\mathcal W_h}\leq C_S \| a(\bd y) \|^{1/2}_{L^\infty(D)}  \|f\|_{L^2(D)}(C_RC_G+1)^{|\bsnu|}(|\bsnu|!)^{\sigma}\boldsymbol b^{\bsnu}.
 \end{equation}
\end{lemma}
\begin{proof}
 Set $\Upsilon_{\bsnu}=\| \der \vec q_h(\vec y) \|_{\mathcal W_h}$. Then
 \begin{align*}
  \Upsilon_{\bd 0} & \le K_0 && \text{by~\Cref{IT:stab},}
  \intertext{with $K_0 = C_S \| a(\bd y) \|_{L^\infty (D)}^{1/2} \| f \|_{L^2(D)}$, while}   
  \Upsilon_{\bsnu} & \le C_R \sum_{\vec 0 \neq \vec m \le \vec \nu} \begin{pmatrix} \vec \nu \\ \vec m \end{pmatrix} \left\| \frac{\derm a(\vec y)}{a(\vec y)} \right\|_{L^\infty(D)} \Upsilon_{\bsnu - \vec m} && \text{by~\Cref{IT:recur},} \\
  & \le  C_R C_G \sum_{\vec 0 \neq \vec m \le \vec \nu} \begin{pmatrix} \vec \nu \\ \vec m \end{pmatrix} ( |\bd m|!)^\sigma {\bd b}^{\bd \nu} \Upsilon_{\bsnu - \vec m} && \text{by~\Cref{IT:gev},} \\
  & \le K_0 (C_RC_G + 1)^{\bd \nu} (|\bsnu|!)^{\sigma}\boldsymbol b^{\bsnu},
 \end{align*}
 by Lemma~\ref{LEM:guth_k} applied with $K_1 = C_R C_G$.
\end{proof}

\begin{lemma}[{see~\cite[Lemma~6.2]{kss12}}] \label{LEM:optimalgamma}
 Given any positive real numbers $\lambda, \rho_i, \beta_i$ and integer $n$, the function
 \[
  g( \gamma_1, \dots, \gamma_n) = \left( \sum_{i=1}^n \gamma_i^\lambda \rho_i \right)^{1/\lambda}  \left( \sum_{i=1}^n\frac{\beta_i}{\gamma_i} \right)
 \]
 is minimized by $\gamma_i = c (\beta_i/\rho_i)^{1/(1+\lambda)}$ for any $c>0$ and the minimum value is 
 \begin{equation}\label{eq:min-g}
  \min g = \left(\sum_{i=1}^n \rho_i^{\frac{1}{1+\lambda}} \beta_i^{\frac{\lambda}{1+\lambda}} \right)^{2/\lambda}.
 \end{equation}
\end{lemma}

\subsection{Gevrey regular model on a bounded parameter domain}\label{SEC:bounded_qmc}
%
In this subsection, we consider the case \( U=[-1/2,1/2]^{\IN} \) with uniform probability measure $\mu({\rm d}\bsy)=\bigotimes_{j\geq 1}{\rm d}y_j$ generated by the probability density function $\vphi$ in \eqref{eq:phiunif}.

For the analysis, we define
\[
 a_{\max} = \sup_{\vec x \in D, \,\vec y\in\corr{U}} [a(\vec y)] (\vec x) \qquad \text{ and } \qquad a_{\min} = \inf_{\vec x \in D,\, \vec y \in \corr{U}} [a(\vec y)] (\vec x).
\]

\begin{theorem}\label{thm:thm1}
 Let Assumptions \ref{IT:norm}--\ref{IT:gev} and the assumptions of Theorem~\ref{THM:bounded} hold with $\mu({\rm d}\bsy)=\bigotimes_{j\geq 1}{\rm d}y_j$ and $U=[-\frac12,\frac12]^{\mathbb N}$. Then there exists a generating vector constructed by the CBC algorithm such that for \( \sigma p < 1 \),
 \[
  \sqrt{\EE \big| I_s^{\varphi}(J(\corr{\vec q_{h}}))-\Qmc(J(\corr{\vec q_{h}}))\big|^2} = \begin{cases} \mathcal O(n^{-r/p+1/2})&\text{if}~p\in (\frac{2r}{3},\frac{1}{\sigma}),\\ \mathcal O(n^{-1+\varepsilon})\; \forall \varepsilon\in(0,\frac{2\sigma r-3}{2\sigma r-4})&\text{if}~p\in (0,\frac{2r}{3}], \end{cases}
 \]
 where the implied coefficient\corr{, which depends on $\varepsilon$, }is independent of $s$ when the generating vector is obtained by a CBC algorithm using the weights
 \begin{equation} \label{eq:gamma-bounded}
  \gamma_{\setu}:=\bigg(\frac{(2\pi^2)^\lambda}{\zeta(2\lambda)}\bigg)^{\frac{|\setu|}{1+\lambda}}\bigg(\sum_{\setv\subseteq\setu}(|\setv|!)^{2r\sigma}\prod_{j\in\setv}(C_RC_G+1)^{2r}b_j^{2r}\bigg)^{\frac{1}{1+\lambda}}
 \end{equation}
 for each $\setu$ in $\{ 1:s\}$, with
 \[
  \lambda=\begin{cases}\frac{p}{2r-p}&\text{if}~p\in(\frac{2r}{3},\frac{1}{\sigma}),\\ \frac{1}{2-2\varepsilon}~\text{for arbitrary}~\varepsilon\in (0,\frac{2\sigma r-3}{2\sigma r-4})&\text{if}~p\in (0,\frac{2r}3].\end{cases}
 \]
\end{theorem}
\begin{remark}
 Although the weights~\eqref{eq:gamma-bounded} are of a general form, an alternative construction given in~\cite{KaarniojaKKNS22} demonstrates how to obtain analogous weights in smoothness-driven product-and-order dependent (SPOD) form guaranteeing dimension-independent QMC convergence rates. The same holds for those weights in Theorem \ref{thm:qmcweight2}.
\end{remark}
\begin{proof}[Proof of Theorem \ref{thm:thm1}]
 We apply Theorem~\ref{THM:bounded} with $F$ set to $ J(\vec q_h)$ to get
 \begin{equation*}%
  \sqrt{\EE\Big[ \left(I_s^\varphi (J(\bd q_h)-\Qmc(J(\bd q_h))\right)^2\Big]} \leq\; \frac{C_{s, \gamma,\lambda}}{\sqrt{R\; n^{1/\lambda}}} \;\|J(\bd q_h)\|_{s,\gamma}.
 \end{equation*}
 The remainder of the proof bounds the norm $\|J(\bd q_h)\|_{s,\gamma},$ as defined in~\eqref{EQ:bounded_norm}. Labeling the squared inner integral there (over the inactive variables not in $\setu$) as $t_\setu$ we estimate it  as follows:
 \begin{align*}
  t_\setu & := \left( \int_{[-1/2,1/2]^{s-|\setu|}} \frac{\partial^{|\setu|}}{\partial \vec y_\setu} J(\vec q_h(\bsy))\, {\rm d}\bsy_{-\setu}\right)^2 \\
  & \le \int_{[-1/2,1/2]^{s-|\setu|}} \bigg|\frac{\partial^{|\setu|}}{\partial \vec y_\setu} J(\vec q_h(\bsy))\bigg|^2 \; \textup d \vec y_{-\setu} && \text{ by H\"older inequality,} \\
  & \le C_J^{2r} \int_{[-1/2,1/2]^{s-|\setu|}}\bigg(\sum_{\mathfrak v\subseteq\setu}\bigg\|\frac{\partial^{|\setv|}}{\partial \vec y_\setv}\vec q_h(\cdot,\bsy)\bigg\|_{L^2(D)}^{r}\bigg)^2 \; \textup d \vec y_{-\setu} && \text{ by \Cref{IT:qoi}}, \\
  & \le 2^{|\setu|}C_J^{{2r}} \int_{[-1/2,1/2]^{s-|\setu|}}\sum_{\mathfrak v\subseteq\setu}\bigg\|\frac{\partial^{|\setv|}}{\partial \vec y_\setv}\vec q_h(\cdot,\bsy)\bigg\|_{L^2(D)}^{2r} \; \textup d \vec y_{-\setu} && \text{ by Young's inequality}, \\
  & \le \frac{2^{|\setu|}C^{2r}_E C_J^{2r}}{a^r_{\min}} \sum_{\mathfrak v\subseteq\setu}\int_{[-1/2,1/2]^{s-|\setu|}} \bigg\|\frac{\partial^{|\setv|}}{\partial \vec y_\setv}\vec q_h(\cdot,\bsy)\bigg\|_{\mathcal W_h}^{2r} \; \textup d \vec y_{-\setu}. && \text{ by \Cref{IT:norm},} \\
  & \le \frac{2^{|\setu|}C^{2r}_E C_J^{2r}}{a^r_{\min}} \sum_{\setv\subseteq\setu}\left( C^2_S a_{\max} \|f\|^2_{L^2(D)}(C_R C_G+1)^{2|\setv|}(|\setv|!)^{2\sigma} \boldsymbol b_\setv^{2} \right)^r, && \text{ by Lemma~\ref{lem:norm-bound-y-derivative},}
 \end{align*}
 since~\Cref{IT:stab}, \Cref{IT:recur} and~\Cref{IT:gev} hold. Here, we use the notation $\boldsymbol b_{\setu}=\prod_{j\in\setu}b_j$ (note that for each subset $\setu\subset\{1,\ldots,s\}$, we can associate a multi-index $\bsnu\in\mathbb N_0^s$ with $\nu_j=1$ for every $j\in \setu$ and $\nu_j=0$ otherwise). Then, setting $C^{2r} = \frac{C^{2r}_E C_J^{2r}}{a^r_{\min}} C^{2r}_S a^r_{\max} \|f\|^{2r}_{L^2(D)}$, and using the value of $C_{s,\gamma,\lambda}$ in \eqref{eq:Csgammalambda},
 \begin{align*}
  C_{s, \gamma, \lambda}^2 & \| J(\bd q_h)\|_{s,\gamma}^2 = C_{s, \gamma, \lambda}^2 \sum_{\setu\subseteq\{1:s\}}\frac{1}{\gamma_{\setu}} \int_{[-1/2,1/2]^{|\setu|}} t_{\setu} \; \dy_{\setu} \\
  & \le \bigg(\sum_{\varnothing\neq\mathfrak u\subseteq\{1:s\}}2\gamma_{\mathfrak u}^{\lambda}\bigg(\frac{2\zeta(2\lambda)}{(2\pi^2)^\lambda}\bigg)^{|\setu|}\bigg)^{\frac{1}{\lambda}} \sum_{\setu\subseteq\{1:s\}}\frac{2^{|\setu|}C^{2r}}{\gamma_{\setu}} \sum_{\setv\subseteq\setu}\left( (C_R C_G+1)^{2|\setv|}(|\setv|!)^{2\sigma} \boldsymbol b_{\setv}^{2} \right)^r.
 \end{align*}
 Letting $\rho_\setu = \Big(\frac{2\zeta(2\lambda)}{(2\pi^2)^\lambda}\Big)^{|\setu|} $ and $\beta_\setu = 2^{|\setu|}\sum_{\setv\subseteq\setu}(C_R C_G+1)^{2r|\setv|}(|\setv|!)^{2r\sigma} \boldsymbol b_{\setv}^{2r}$, we may apply Lemma~\ref{LEM:optimalgamma} to minimize the above upper bound by choosing $\gamma_\setu$ as in~\eqref{eq:gamma-bounded}. Then we obtain
 \[
  C_{s, \gamma, \lambda}\| J(\bd q_h)\|_{s,\gamma} \le C^r 2^{1/\lambda}
  M_{s, \gamma, \lambda}
 \]
 where $M_{s, \gamma, \lambda}$ is the minimized upper bound---see~\eqref{eq:min-g}--given by
 \begin{align*}
  M_{s,\boldsymbol\gamma, \lambda }^\lambda & =\sum_{\mathfrak u\subseteq\{1:s\}}
  \rho_\setu^{\frac{1}{1+\lambda}} \beta_\setu^{\frac{\lambda}{1+\lambda}} \\
  & = \sum_{\mathfrak u\subseteq\{1:s\}} \rho_\setu^{\frac{1}{1+\lambda}} 2^{\frac{\lambda}{1+\lambda}|\setu|}\bigg(\sum_{\setv\subseteq\setu} (|\mathfrak v|!)^{ {2\sigma r } } (C_RC_G+1)^{{2r|\setv|}} \prod_{j\in \setv}b_j^{{2 r}}\bigg)^{\frac{\lambda}{1+\lambda}}\\
  & \leq \sum_{\mathfrak u\subseteq\{1:s\}} \rho_\setu^{\frac{1}{1+\lambda}} 2^{\frac{\lambda}{1+\lambda}|\setu|}\sum_{\setv\subseteq\setu} (|\mathfrak v|!)^{\frac {2\lambda\sigma r }{1+\lambda}} (C_RC_G+1)^{\frac{2\lambda r|\setv|}{1+\lambda}} \prod_{j\in \setv}b_j^{\frac{2\lambda r}{1+\lambda}}\\
  &=\sum_{\mathfrak u\subseteq\{1:s\}} \bigg(\frac{2\zeta(2\lambda)}{(2\pi^2)^{\lambda}} \bigg)^{\frac{|\mathfrak u|}{1+\lambda}}2^{\frac{\lambda}{1+\lambda}|\setu|}\sum_{\boldsymbol m_{\setu}\in \{0:1\}^{|\setu|}} (|\boldsymbol m_{\setu}|!)^{\frac {2\lambda\sigma r }{1+\lambda}} (C_RC_G+1)^{\frac{2\lambda r|\boldsymbol m_{\setu}|}{1+\lambda}} \prod_{j\in \setu}b_j^{\frac{2\lambda r}{1+\lambda}{m_\setu}_j}\\
  &\leq\sum_{\mathfrak u\subseteq\{1:s\}}\sum_{\boldsymbol m_{\setu}\in \{0:1\}^{|\setu|}} \bigg((|\boldsymbol m_{\setu}|!)^{ {\sigma r }} \prod_{j\in \setu}\Xi_j^{r m_j}\bigg)^{\frac{2\lambda}{1+\lambda}},
 \end{align*}
 where $\Xi_j = \max\big\{1,2^{1/(2r)}\big(\frac{2\zeta(2\lambda)}{(2\pi^2)^\lambda}\big)^{1/(2r\lambda)}\big\}(C_RC_G+1)b_j$. The first inequality is Jensen's inequality in the version of \cite[Thm.\ 19]{HardyLP53}, and the second inequality follows from the definition of the \( \Xi_j \). Noting that
 \[
  \sum_{\setu\subseteq\{1:s\}}\sum_{\boldsymbol m_{\setu}\in\{0:1\}^{|\setu|}}a_{\setu}=\sum_{\setu\subseteq\{1:s\}}2^{s-|\setu|}a_{\setu}\leq \sum_{\setu\subseteq\{1:s\}}2^{s}a_{\setu}
 \]
 and setting \( \Psi_j:=2\sqrt 2 \Xi_j \), we can rewrite the resulting upper bound as
 \begin{equation}
  M_{s,\boldsymbol\gamma,\lambda}^{\lambda}\leq \sum_{\bsnu\in\{0:1\}^{s}}\bigg((|\bsnu|!)^{\sigma r}\prod_{j=1}^s 2^{\frac{1+\lambda}{2\lambda}}\Xi_j^{r\nu_j}\bigg)^{\frac{2\lambda}{1+\lambda}}\leq \sum_{\bsnu\in\{0:1\}^{s}}\bigg((|\bsnu|!)^{\sigma r}\prod_{j=1}^s \Psi_j^{r\nu_j}\bigg)^{\frac{2\lambda}{1+\lambda}}.\label{eq:jensenbit}
 \end{equation}
 Here, we use an argument from \cite{spodpaper14}: We recast this sum as a sum over subsets as
 \begin{align}\notag
  \sum_{\bsnu\in\{0:1\}^{s}}\bigg((|\bsnu|!)^{\sigma r}\prod_{j=1}^s \Psi_j^{r\nu_j}\bigg)^{\frac{2\lambda}{1+\lambda}}&\leq \sum_{\substack{\setu\subseteq\mathbb N\\ |\setu|<\infty}}\bigg((|\setu|!)^{\sigma r}\prod_{j\in\setu}\Psi_j^r\bigg)^{\frac{2\lambda}{1+\lambda}} && \text{ (more summands)}\\
  &\leq \sum_{\substack{\setu\subseteq\mathbb N\\ |\setu|<\infty}}(|\setu|!)^{\frac{2\lambda \sigma r}{1+\lambda}}\prod_{j\in\setu}\Psi_j^{\frac{2\lambda r}{1+\lambda}} && \text{ (Jensen's inequality)}\notag\\
  &= \sum_{\ell=0}^\infty (\ell!)^{\frac{2\lambda \sigma r} {1+\lambda}}\sum_{\substack{\setu\subseteq\mathbb N\\ |\setu|=\ell}}\prod_{j\in\setu}\Psi_j^{\frac{2\lambda r}{1+\lambda}}.\notag
 \end{align}

 Now we use the inequality
 \[
  \bigg(\sum_{j=1}^\infty c_j\bigg)^\ell=\sum_{j_1=1}^\infty \cdots \sum_{j_\ell=1}^\infty c_{j_1}\cdots c_{j_\ell}\geq \ell! \sum_{\substack{|\setu|=\ell\\ \setu\subset \mathbb N}}\prod_{j\in\setu}c_j,
 \]
 where we take into account the fact that each term $\prod_{j\in\setu}c_j$ with $|\setu|=\ell$ appears $\ell!$ times in the iterated sum on the left-hand side. This yields
 \[
  \sum_{\bsnu\in\{0:1\}^{s}}\bigg((|\bsnu|!)^{\sigma r}\prod_{j=1}^s \Psi_j^{r\nu_j}\bigg)^{\frac{2\lambda}{1+\lambda}}\leq \sum_{\ell=0}^\infty (\ell!)^{\frac{2\lambda \sigma r}{1+\lambda}-1}\bigg(\sum_{j\geq 1}\Psi_j^{\frac{2\lambda r}{1+\lambda}}\bigg)^\ell \corr{.}
 \]

 There are two cases to consider:

 \begin{enumerate}
  \item If $p\in(\frac{2r}{3},\frac{1}{\sigma})$, we can choose $\lambda=\frac{p}{2r-p}$. This implies
  \[
   M_{s,\boldsymbol \gamma,\lambda}^\lambda\leq \sum_{\ell=0}^\infty \underset{=:a_{\ell}}{\underbrace{(\ell!)^{\sigma p -1}\bigg(\sum_{j=1}^\infty \Psi_j^{p}\bigg)^\ell}}.
  \]
  By the ratio test, recalling that $\sigma p<1$, we obtain
  \[
   \frac{a_{\ell+1}}{a_\ell}=(\ell+1)^{\sigma p-1}\sum_{j=1}^\infty \Psi_j^{p}\xrightarrow{\ell\to\infty}0,
  \]
  since $\boldsymbol \Psi\in \ell^p$ if and only if $\boldsymbol b\in \ell^p$. This results in dimension-independent QMC convergence with rate $\mathcal O(n^{-\frac{r}{p}+\frac12})$.

  \item If $p\in(0,\min\{\frac{2r}{3},\frac{1}{\sigma}\}]$, we additionally assume that $r\sigma<\frac32$ and let $\lambda=\frac{1}{2-2\varepsilon}$ for arbitrary $\varepsilon\in (0,\frac{2\sigma r-3}{2\sigma r-4})$. Now $\frac{2\lambda}{1+\lambda}=\frac{2}{3-2\varepsilon}>\frac{2}{3}.$ Now
  \[
   M_{s,\boldsymbol\gamma,\lambda}^{\lambda}\leq \sum_{\ell=0}^\infty (\ell!)^{\frac{2\sigma r}{3-2\varepsilon}-1}\bigg(\sum_{j=1}^\infty \Psi_j^{\frac{2}{3-2\varepsilon}r}\bigg)^\ell.
  \]
  Now there holds $\frac{2}{3-2\varepsilon}r> p$ since $p<\frac{2r}{3}$. We can now use Jensen's inequality $\sum_j c_j\leq \big(\sum_j c_j^{\mu}\big)^{1/\mu}$ for $c_j\geq 0$ and $\mu\in(0,1]$ with $\mu=\frac{3-2\varepsilon}{2}\frac{p}{r}$ to deduce that
  \[
   M_{s,\boldsymbol\gamma,\lambda}^{\lambda}\leq\sum_{\ell=0}^\infty (\ell!)^{\frac{2\sigma r}{3-2\varepsilon}-1}\bigg(\sum_{j=1}^\infty \Psi_j^{p}\bigg)^{\ell \frac{2}{3-2\varepsilon}\frac{r}{p}}.
  \]
  Similarly to the other case, the ratio test implies that this upper bound is finite as long as $\frac{2}{3-2\varepsilon}\sigma r<1$, which especially holds if $\sigma r < \frac32$.

  \item If $p\in(0,\min\{\frac{2r}{3},\frac1{\sigma}\}]$ and $r\sigma\geq \frac32$, we introduce $\widetilde\sigma=\delta \sigma$, where $\delta\in [\frac{1}{r\sigma},\frac3{2r\sigma})$ and apply Jensen's inequality to~\eqref{eq:jensenbit} which yields
  \[
   M_{s,\boldsymbol\gamma,\lambda}^{\lambda} \leq \bigg(\sum_{\bsnu\in\{0:1\}^{s}}\bigg((|\bsnu|!)^{\delta\sigma r}\prod_{j=1}^s \Psi_j^{\delta r\nu_j}\bigg)^{\frac{2\lambda}{1+\lambda}}\bigg)^{1/\delta} \leq \bigg(\sum_{\ell=0}^\infty (\ell!)^{\frac{2\lambda\widetilde\sigma r}{1+\lambda}-1}\bigg(\prod_{j=1}^s \Psi_j^{\frac{2\lambda \delta r}{1+\lambda}}\bigg)^{\ell}\bigg)^{1/\delta}
  \]
  by following the same steps as before. By choosing $\lambda=\frac1{2-2\varepsilon}$, the argument from case 2 applies and we infer that the resulting convergence rate $\mathcal O(n^{-\frac{r}{p}+\frac12})$ is independent of the dimension.
\end{enumerate}
\end{proof}

\begin{example}[uniform and affine model] 
 In the uniform and affine model (cf., e.g.,~\cite{cds10,spodpaper14,dicklegiaschwab,ghs18,SchwabG11,kss12,kssmultilevel,schwab13}), the uncertain coefficient is endowed with the parameterization
 \[
  [a(\bsy)](\vec x)=a_0(\vec x)+\sum_{j=1}^\infty y_j\psi_j(\vec x)\qquad \text{for all}~\vec x\in D,
 \]
 where $y_1,y_2,\ldots$ are independently and identically distributed uniform random variables supported in $[-1/2,1/2]$, where we denote $\bsy=(y_1,y_2,\ldots)\in U$ with $U=[-1/2,1/2]^{\mathbb N}$. We assume that $a_0\in L^\infty(D)$ and $\psi_j\in L^\infty(D)$ for all $j\geq 1$ are such that $\sum_{j\geq 1}\|\psi_j\|_{L^\infty(D)}^p<\infty$ for some $p\in(0,1)$ and suppose that there exist constants $a_{\min},a_{\max}>0$ such that $a_{\min}\leq [a(\bsy)](\vec x)\leq a_{\max}$ for all $x\in D$ and $\bsy\in U$. By defining $\boldsymbol b=(b_1,b_2,\ldots)$ as $b_j=\|\psi_j\|_{L^\infty(D)}/a_{\min}$, it follows that
 \[
  \bigg\|\frac{\partial_{\bsy}^{\bsnu}a(\bsy)}{a(\bsy)}\bigg\|_{L^\infty(D)}\leq \prod_{j\in{\rm supp}(\bsnu)}b_j,
 \]
meaning that our theory covers this framework.
\end{example}

\subsection{Unbounded and Gevrey regular model}
%
Next we investigate a model for the input random field, which is Gevrey $\sigma$ regular subject to parameters $\bsy=(y_1,y_2,\ldots)\in U$ with {\em unbounded} support $U=\mathbb R^{\mathbb N}$ and a Gaussian product probability measure $\mu=\bigotimes_{j\geq 1}\mathcal N(0,1)$. 

For the analysis, we use the notation
\begin{equation*} 
 \overline a(\vec y) = \| a(\vec y) \|_{L^\infty(D)} \qquad \text{ and } \qquad \underline a(\vec y) = \| a^{-1}(\vec y) \|^{-1}_{L^\infty(D)}.
\end{equation*}

Lemmas \ref{lem:norm-bound-y-derivative} and \ref{LEM:guth_k} with $K_0= C_S \sqrt{\overline a (\vec y)} \| f \|_{L^2(D)}$, $K_1=C_M$, and $\Upsilon_{\bsnu}=\| \der \vec q_h(\vec y) \|_{\mathcal W_h}$ yield
\[
\| \der \vec q_h(\vec y) \|_{\mathcal W_h}\leq C_S\sqrt{\overline{a}(\bsy)}\|f\|_{\mesh}(C_M+1)^{|\bsnu|}(|\bsnu|!)^{\sigma}\boldsymbol b^{\bsnu}.
\]
This bound can be used to control the weighted Sobolev norm~\eqref{eq:gaussnorm} of $\|\vec q_h(\vec y)\|_{\mathcal W_h}$.

The following theorem states the QMC integration error for the dimensionally truncated solution, in this setting, if the following additional assumption holds.
\begin{assumption}\label{IT:rat}
 \textbf{Ratio bound of the parameter.} There is a constant \( C_Q  > 0 \) such that the quotient of the maximum and the minimum values of \( a(\vec y) \) is bounded, i.e., since \( \inf_{\vec x \in D} a(\vec x, \vec y) = \| a^{-1}(\vec y) \|^{-1}_{L^\infty(D)} \),
 \[ \| a(\vec y) \|_{L^\infty(D)}^r \; \| a^{-1}(\vec y) \|_{L^\infty(D)}^r \le C_Q \prod_{j = 1}^\infty \exp\left( 2b_j|y_j| \right)\quad\corr{\text{for all}~\bsy\in U} \]
 with the \( b_j \) as characterized in \Cref{IT:gev} and \( r \) as in \Cref{IT:qoi}.
\end{assumption}
\corr{Notably, while \Cref{IT:gev} is used to establish dimension-independent QMC convergence rates, \Cref{IT:rat} is needed to ensure that the expected value of the PDE response does not diverge. We remark that Assumption~\ref{IT:rat} needs to be enforced in the infinite-dimensional parameter space $U$.}

\begin{theorem}\label{thm:qmcweight2}
 Let Assumptions \ref{IT:norm}--\ref{IT:rat} hold with $\mu=\bigotimes_{j\geq 1}\mathcal N(0,1)$ and $U=\mathbb R^{\mathbb N}$. Then, there exists a generating vector constructed by the CBC algorithm such that for \( \sigma p < 1 \),
 \[
  \sqrt{\EE \big|I_s^{\varphi}(J(\corr{\vec q_{h}}))-\Qmc(J(\corr{\vec q_{h}}))\big|^2} =\begin{cases} \mathcal O(n^{-r/p+1/2})&\text{if}~p\in (\frac{2r}3,\frac1\sigma),\\ \mathcal O(n^{-1+\varepsilon})\; \forall \varepsilon\in(0,\frac{2\sigma r-3}{2\sigma r-4})&\text{if}~p\in (0,\frac{2r}{3}], \end{cases}
 \]
 where the implied coefficient\corr{, which depends on $\varepsilon$, } is independent of $s$ when the generating vector is obtained by a CBC algorithm using the weights
 \[
  \gamma_{\setu}=\bigg(\frac{2^{|\setu|}\sum_{\setv\subseteq\setu}((C_RC_G+1)^{2r|\setv|}(|\setv|!)^{2r\sigma}\boldsymbol b_{\setv}^{2r})}{\prod_{j\in\setu}\varrho_j(\lambda)2{\rm e}^{2b_j^2}\Phi(2b_j)(\alpha_j-b_j)}\bigg)^{\frac{1}{1+\lambda}}, \quad \text{ where } \quad
 \alpha_j=\frac12\bigg(\beta_j+\sqrt{\beta_j^2+1-\frac{1}{2\lambda}}\bigg)
 \]
 and
 \begin{equation*}
  \lambda=\begin{cases} \frac{p}{2r-p}&\text{if}~p\in (\frac{2r}{3},\frac1\sigma),\\
 \frac{1}{2-2\varepsilon}~\text{for arbitrary}~\varepsilon\in (0,\frac{2\sigma r-3}{2\sigma r-4})&\text{if}~p\in (0,\frac{2r}{3}]. \end{cases}%
 \end{equation*}
\end{theorem}
\begin{proof}
 We apply Theorem~\ref{thm:lognormalqmcerror} with $F$ set to $J(\boldsymbol q_h)$ to get
 \[
  \sqrt{\mathbb E\big[(I_s^{\varphi}(J(\boldsymbol q_h))-Q_{\boldsymbol\Delta}^{\varphi}(J(\boldsymbol q_h)))^2\big]}\leq\frac{C_{s,\boldsymbol\gamma,\lambda,\alpha}}{\sqrt{R\,n^{1/2}}}\|F\|_{s,\boldsymbol\gamma,\alpha}. 
 \]
 We begin by bounding the norm $\|J(\boldsymbol q_h)\|_{s,\boldsymbol\gamma,\alpha}$ defined in~\eqref{eq:gaussnorm}. Similarly to the affine and uniform case, we denote by $t_{\setu}$ the inner integral over the inactive variables not in $\setu$ and proceed to estimate
 \begin{align*}
  t_{\setu}&:=\bigg(\int_{\mathbb R^{s-|\setu|}}\frac{\partial^{|\setu|}}{\partial\bsy_{\setu}}J(\boldsymbol q_h(\bsy))\prod_{j\in\{1:s\}\setminus\setu}\varphi(y_j)\,{\rm d}\bsy_{-\setu}\bigg)^2\\
  &\leq \bigg(\int_{\mathbb R^{s-|\setu|}}\bigg|\frac{\partial^{|\setu|}}{\partial\bsy_{\setu}}J(\boldsymbol q_h(\bsy))\bigg|^2\prod_{j\in\{1:s\}\setminus \setu}\varphi(y_j)\,{\rm d}\bsy_{-\setu}\bigg) && \text{ by Hölder's ineq.}\\
  & \qquad \times \bigg(\int_{\mathbb R^{s-|\setu|}}\prod_{j\in\{1:s\}\setminus \setu}\varphi(y_j)\,{\rm d}\bsy_{-\setu}\bigg) && \text{ $=1$}\\
  &\leq C_J^{2r}\int_{\mathbb R^{s-|\setu|}}\bigg(\sum_{\setv\subseteq\setu}\bigg\|\frac{\partial^{|\setv|}}{\partial \bsy_{\setv}}\boldsymbol q_h(\bsy)\bigg\|_{L^2(D)}^r\bigg)^2\prod_{j\in\{1:s\}\setminus \setu}\varphi(y_j)\,{\rm d}\bsy_{-\setu} && \text{ by \Cref{IT:qoi}}\\
  &\leq C_J^{2r}C_E^{2r}\int_{\mathbb R^{s-|\setu|}}\frac{1}{{\underline{a}(\bsy)^r}}\bigg(\sum_{\setv\subseteq\setu}\bigg\|\frac{\partial^{|\setv|}}{\partial \bsy_{\setv}}\boldsymbol q_h(\bsy)\bigg\|_{\mathcal W_h}^r\bigg)^2\prod_{j\in\{1:s\}\setminus \setu}\varphi(y_j)\,{\rm d}\bsy_{-\setu} && \text{ by \Cref{IT:norm},}\\
  &\leq 2^{|\setu|}C_J^{2r}C_E^{2r}\int_{\mathbb R^{s-|\setu|}}\frac{1}{{\underline{a}(\bsy)^r}}\sum_{\setv\subseteq\setu}\bigg\|\frac{\partial^{|\setv|}}{\partial \bsy_{\setv}}\boldsymbol q_h(\bsy)\bigg\|_{\mathcal W_h}^{2r}\prod_{j\in\{1:s\}\setminus \setu}\varphi(y_j)\,{\rm d}\bsy_{-\setu} && \text{ by Young's ineq.}\\
  &\leq 2^{|\setu|}C_J^{2r}C_E^{2r}\int_{\mathbb R^{s-|\setu|}}\frac{\overline{a}(\bsy)^{r}}{{\underline{a}(\bsy)^r}}\sum_{\setv\subseteq\setu}\big(C_S^{2r}\|f\|_{L^2(D)}^{2r}(C_RC_G+1)^{2r|\setv|} && \text{ by Lemma~\ref{lem:norm-bound-y-derivative},} \\
  & \qquad \times (|\setv|!)^{2r\sigma}\boldsymbol b_{\setv}^{2r}\big)\prod_{j\in\{1:s\}\setminus \setu}\varphi(y_j)\,{\rm d}\bsy_{-\setu}.
 \end{align*}
 Then, setting $C^{2r}=C_E^{2r}C_J^{2r}C_S^{2r}\|f\|_{L^2(D)}^{2r}$ and using the value of $C_{s,\boldsymbol\gamma,\lambda,\alpha}$ in \eqref{EQ:cgla_log} we obtain
 \begin{align*}
  C_{s,\boldsymbol\gamma,\lambda,\alpha}^2\|J(\boldsymbol q_h) & \|_{s,\boldsymbol\gamma,\alpha}^2 = C_{s,\boldsymbol\gamma,\lambda,\alpha}^2\sum_{\setu\subseteq\{1:s\}}\frac{1}{\gamma_{\setu}}\int_{\mathbb R^{|\setu|}}t_{\setu}\prod_{j\in\setu}\varpi_j^2(y_j)\,{\rm d}\bsy_{\setu}\\
  &\leq \bigg(2\sum_{\varnothing\neq \setu\subseteq\{1:s\}}\gamma_{\setu}^{\lambda}\prod_{j\in\setu}\varrho_j(\lambda)\bigg)^{\frac{1}{\lambda}} \sum_{\setu\subseteq\{1:s\}}\frac{2^{|\setu|}C^{2r}}{\gamma_{\setu}}\sum_{\setv\subseteq\setu}\big((C_RC_G+1)^{2r|\setv|}(|\setv|!)^{2r\sigma}\boldsymbol b_{\setv}^{2r}\big)\\
  & \qquad \times\int_{\mathbb R^{s}}\frac{\overline{a}(\bsy)^{r}}{{\underline{a}(\bsy)^r}}\prod_{j\in\{1:s\}\setminus \setu}\varphi(y_j)\prod_{j\in\setu}\varpi_j^2(y_j)\,{\rm d}\bsy
 \end{align*}
 Using \Cref{IT:rat}, Fubini's theorem, and the integral identity $\int_{\mathbb R}{\rm e}^{2b_j|y|}\varphi(y)\,{\rm d}y=2{\rm e}^{2b_j^2}\Phi(2b_j)$, we receive
 \begin{align*}
  \int_{\mathbb R^s}\frac{\overline{a}(\bsy)^r}{\underline{a}(\bsy)^r}&\prod_{j\in\{1:s\}\setminus\setu}\varphi(y_j)\prod_{j\in\setu}\varpi_j^2(y_j)\,{\rm d}\bsy\\
  &\leq C_Q\int_{\mathbb R^s}\bigg(\prod_{j\in\{1:s\}\setminus\setu}{\rm e}^{2b_j|y_j|}\varphi(y_j)\bigg)\bigg(\prod_{j\in\setu}{\rm e}^{2b_j|y_j|}\varpi_j^2(y_j)\bigg)\,{\rm d}\bsy && \text{ by \Cref{IT:rat}}\\
  &=C_Q\prod_{j\in\{1:s\}\setminus\setu}\int_{\mathbb R}{\rm e}^{2b_j|y_j|}\varphi(y_j)\,{\rm d}y_j \prod_{j\in\setu}\int_{\mathbb R} {\rm e}^{2b_j|y_j|}\varpi_j^2(y_j)\,{\rm d}y_j && \text{ by Fubini}\\
  &= C_Q \prod_{j=1}^s 2{\rm e}^{2b_j^2}\Phi(2b_j)\prod_{j\in\setu}\frac{1}{2{\rm e}^{2b_j^2}\Phi(2b_j)}\int_{\mathbb R}{\rm e}^{2b_j|y_j|}\varpi_j^2(y_j)\,{\rm d}y_j && \text{ by integral ident.}
 \end{align*}
 Recalling that $\varpi_j(y_j)={\rm e}^{-\alpha_j|y_j|}$, the remaining integral is finite provided that $\alpha_j>b_j$, with 
 \[
  \int_{\mathbb R}{\rm e}^{2b_j|y_j|}\varpi_j^2(y_j)\,{\rm d}y_j=\frac{1}{\alpha_j-b_j}
 \]
 and we obtain
 \begin{multline*}
  C_{s,\boldsymbol\gamma,\lambda,\alpha}^2\|J(\boldsymbol q_h)\|_{s,\boldsymbol\gamma,\alpha}^2 \leq \bigg(2\sum_{\varnothing\neq \setu\subseteq\{1:s\}}\gamma_{\setu}^{\lambda}\prod_{j\in\setu}\varrho_j(\lambda)\bigg)^{\frac{1}{\lambda}} C_Q\\
  \sum_{\setu\subseteq\{1:s\}}\frac{2^{|\setu|}C^{2r}}{\gamma_{\setu}}\sum_{\setv\subseteq\setu}\big((C_RC_G+1)^{2r|\setv|}(|\setv|!)^{2r\sigma}\boldsymbol b_{\setv}^{2r}\big) \prod_{j=1}^s 2{\rm e}^{2b_j^2}\Phi(2b_j)\prod_{j\in\setu}\frac{1}{2{\rm e}^{2b_j^2}\Phi(2b_j)(\alpha_j-b_j)}.
 \end{multline*}
 Similarly to the affine and uniform case, the upper bound is (up to a constant scaling factor) of the form
 \[
  \boldsymbol \gamma\mapsto \bigg(\sum_i \rho_i\gamma_i^{\lambda}\bigg)^{1/\lambda}\bigg(\sum_i \beta_i\gamma_i^{-1}\bigg),
 \]
 where $\rho_{\setu}=\prod_{j\in\setu}\varrho_j(\lambda)$ and \[\beta_{\setu}=2^{|\setu|}\sum_{\setv\subseteq\setu}((C_RC_G+1)^{2r|\setv|}(|\setv|!)^{2r\sigma}\boldsymbol b_{\setv}^{2r})\prod_{j\in\setu}\frac{1}{2{\rm e}^{2b_j^2}\Phi(2b_j)(\alpha_j-b_j)}.\] The minimizing weights are given by
 \[
  \gamma_{\setu}=\bigg(\frac{2^{|\setu|}\sum_{\setv\subseteq\setu}((C_RC_G+1)^{2r|\setv|}(|\setv|!)^{2r\sigma}\boldsymbol b_{\setv}^{2r})}{\prod_{j\in\setu}\varrho_j(\lambda)2{\rm e}^{2b_j^2}\Phi(2b_j)(\alpha_j-b_j)}\bigg)^{\frac{1}{1+\lambda}}
 \]
 Now plugging these optimized weights into the error bound yields
 \begin{align*}
  M_{s,\boldsymbol\gamma,\lambda,\alpha}^{\lambda}&=\sum_{\setu\subseteq\{1:s\}}\rho_{\setu}^{\frac{1}{1+\lambda}}\beta_{\setu}^{\frac{\lambda}{1+\lambda}}\\
  &=\sum_{\setu\subseteq\{1:s\}}\prod_{j\in\setu}\varrho_j(\lambda)^{\frac{1}{1+\lambda}}2^{\frac{\lambda}{1+\lambda}|\setu|}\bigg(\prod_{j\in\setu}\frac{1}{2{\rm e}^{2b_j^2}\Phi(2b_j)(\alpha_j-b_j)}\bigg)^{\frac{\lambda}{1+\lambda}} \\
  & \qquad \times \bigg(\sum_{\setv\subseteq\setv}((C_RC_G+1)^{2r|\setv|}(|\setv|!)^{2r\sigma}\boldsymbol b_{\setv}^{2r})\bigg)^{\frac{\lambda}{1+\lambda}}.
 \end{align*}
 Choosing \( \alpha_j=\frac12\left(b_j+\sqrt{b_j^2+1-\frac{1}{2\lambda}}\right) \) ensures that $\frac{\partial}{\partial \alpha_j}M_{s,\boldsymbol\gamma,\lambda,\alpha}^{\lambda}=0$ for all $j$, hence choosing $\alpha_j$ this way minimizes the constant factor in the error bound. Moreover, upper bounding via Jensen's inequality yields
 \begin{align*}
  M_{s,\boldsymbol\gamma,\lambda,\alpha}^{\lambda} &\leq\sum_{\setu\subseteq\{1:s\}}\prod_{j\in\setu}\varrho_j(\lambda)^{\frac{1}{1+\lambda}}2^{\frac{\lambda}{1+\lambda}|\setu|}\bigg(\prod_{j\in\setu}\frac{1}{2{\rm e}^{2b_j^2}\Phi(2b_j)(\alpha_j-b_j)}\bigg)^{\frac{\lambda}{1+\lambda}}\\
  & \qquad \times \sum_{\setv\subseteq\setu}(C_RC_G+1)^{\frac{2\lambda r|\setv|}{1+\lambda}}(|\setv|!)^{\frac{2\lambda r\sigma}{1+\lambda}}\boldsymbol b_{\setv}^{\frac{2\lambda r}{1+\lambda}}\\
  &=\sum_{\setu\subseteq\{1:s\}}\bigg(\prod_{j\in\setu} \rho_j(\lambda)^{\frac{1}{1+\lambda}}\bigg)2^{\frac{\lambda}{1+\lambda}|\setu|}\bigg(\prod_{j\in\setu}\frac{1}{2{\rm e}^{2b_j^2}\Phi(2b_j)(\alpha_j-b_j)}\bigg)^{\frac{\lambda}{1+\lambda}}\\
  &\qquad \times \sum_{\boldsymbol m_{\setu}\in\{0:1\}^{|\setu|}}(C_RC_G+1)^{\frac{2\lambda r |\boldsymbol m_{\setu}|}{1+\lambda}}(|\boldsymbol m_{\setu}|)!^{\frac{2\lambda r\sigma}{1+\lambda}}\prod_{j\in\setu}b_j^{\frac{2\lambda r}{1+\lambda}m_j}\\
  &\leq \sum_{\setu\subseteq\{1:s\}}\sum_{\boldsymbol m_{\setu}\in\{0:1\}^{|\setu|}}\bigg((|\boldsymbol m_{\setu}|!)^{\sigma r}\prod_{j\in\setu}\Xi_j^{rm_j}\bigg)^{\frac{2\lambda}{1+\lambda}},
 \end{align*}
 where the last inequality holds, since we define
 \[
  \Xi_j=\max\{1,\rho_j(\lambda)^{\frac{1}{2}}2^{\frac{1}{2}}\bigg(\frac{1}{2{\rm e}^{2b_j^2}\Phi(2b_j)(\alpha_j-b_j)}\bigg)^{\frac{1}{2}}\}^{1/r}(C_RC_G+1)b_j.
 \]
 Noting that
 \[
  \frac{\rho_j(\lambda)}{\alpha_j-b_j}\propto \frac{\exp\big(\frac{\lambda}{4\eta}(b_j+\sqrt{b_j^2+1-\frac{1}{2\lambda}})^2)}{\frac12\sqrt{b_j^2+1-\frac{1}{2\lambda}}-\frac12 b_j}=:g_{\lambda}(b_j).
 \]
 Evidently $\frac{\partial}{\partial b}g_{\lambda}(b)=\frac{16{\rm e}^{\frac14(2b+\sqrt{2+4b^2})^2}}{(\sqrt{4b^2+2}-2b)^2}>0$, and since $(b_j)_{j\geq 1}$ is summable, there exists a number $C>0$ independently of $s$ but depending on $\lambda$ such that \( \Xi_j\leq Cb_j \) and hence $(\Xi_j)\in\ell^p$ for the same $p$ satisfying $(b_j)\in\ell^p$. The remainder of the proof follows by following the steps taken in the proof of the affine case.
\end{proof}

\begin{example}[lognormal model]
 In the lognormal case  (cf., e.g.,~\cite{gittelson,GrahamKNSSS15,GrahamKNSS11,log3,log4,log5,schwabtodor}), we define
 \begin{equation*}
  [a(\vec y)](\vec x) = a_0(\vec x) \exp\left( \sum_{j=1}^\infty y_j \psi_j(\vec x) \right) \qquad \text{ for all } \vec x \in D,
 \end{equation*}
 where $y_1, y_2, \dots$ are independently and identically distributed standard normal random variables, composing $\vec y = (y_1, y_2, \dots)$, and $a_0(\vec x) > 0$. Defining $\vec b = (b_1, b_2, \dots)$ via $b_j = \| \psi_j \|_{L^\infty(D)}$, we assume that $\vec b$ satisfies $\sum_{j \in \IN} b^p_j < \infty$ for some $p \in (0,1]$, which allows us to define the admissible set of stochastic parameters $\vec y$ as
 \begin{equation*}
  U_{\vec b} = \left\{ \vec y \in \IR^\IN\colon \sum_{j \in \IN} b_j |y_j| < \infty \right\} \subset \IR^\IN.
 \end{equation*}
 Although $U_{\vec b}$ is not a countable product of subsets of $\IR$, it is still $\vec \mu_G$ measurable and of full Gaussian measure, i.e., $\vec \mu_G(U_{\vec b}) = 1$, see \cite[Lem.\ 2.28]{SchwabG11}. In this setting, there holds
 $$
\corr{\partial_{\bsy}^{\bsnu}a(\bsy)=a(\bsy)\prod_{j\in{\rm supp}(\bsnu)}\psi_j(\bsx)^{\nu_j}}
 $$
 and hence
 \begin{equation*}
  \left\| \frac{\partial_{\bsy}^{\bsnu} a(\vec y)}{a(\vec y)} \right\|_{L^\infty(D)} {\corr{\leq}\,} \boldsymbol b^{\bsnu},%
 \end{equation*}
\corr{showing that the lognormal setting satisfies the ratio bound in Assumption~\ref{IT:gev}.}%
\end{example}

\section{Finite element methods that fulfill Assumptions \ref{IT:norm}--\ref{IT:recur}}\label{SEC:finite_elements}
%
\corr{The purpose of this section is to demonstrate how the abstract assumptions introduced in Section \ref{SEC:setting} can be verified for concrete finite element discretizations of the balance law. Rather than treating these assumptions as purely formal conditions, it is useful to interpret them operationally in terms of familiar ingredients from finite element analysis. \Cref{IT:norm} expresses that the discrete norm \( \|\cdot\|_{\mathcal W_h} \) controls the physical energy of the flux, i.e., that the chosen method provides an energy norm dominating the \( L^2(D) \) norm of the flux weighted by the coefficient. In practice this follows from standard boundedness properties of the bilinear forms defining the method and from the definition of the discrete norm. \Cref{IT:stab} represents the stability of the discretization: the discrete flux must remain bounded in the \( \mathcal W_h \)-norm in terms of the data and the coefficient. For mixed and hybridizable methods this typically follows from a discrete inf–sup condition together with boundedness of the coefficient field. \Cref{IT:recur} is the key parametric derivative recursion required for the QMC analysis. Operationally, it is obtained by differentiating the discrete equations with respect to the parameter variables and using stability of the discrete problem to bound the resulting terms; the constants appearing in the recursion correspond directly to the stability constant of the method and the bounds on derivatives of the coefficient in \Cref{IT:gev}}

\corr{With this interpretation in mind, verifying the sufficient conditions reduces to three standard steps for each numerical scheme: (i) identify the discrete flux space \( \mathcal W_h \) and norm \( \|\cdot\|_{\mathcal W_h} \) so that \Cref{IT:norm} holds; (ii) establish stability of the discrete formulation, typically via an inf–sup argument or coercivity estimate yielding \Cref{IT:stab}; and (iii) differentiate the discrete equations with respect to the parameters and apply the stability estimate to obtain the recursive bound \Cref{IT:recur} The remaining assumptions involve either the quantity of interest or the coefficient model and are therefore largely method-independent. The derivations below illustrate these steps for several representative discretizations—including conforming finite elements, mixed methods, and hybridizable discontinuous Galerkin schemes—while explicitly identifying how their stability constants and norms enter the abstract framework.}

Let us consider a mesh $\mesh \!=\! \{ \elem_1, \elem_2, \dots, \elem_N \}$. The union of the element boundaries is denoted by $\partial \mesh = \bigcup_{\elem \in \mesh} \partial \elem$. For functions $u_h, v_h \in \prod_{E \in \mesh} L^2(E), \vec q_h, \vec r_h \in \prod_{E \in \mesh} L^2(E)^d$,  and $m_h, \mu_h \in \prod_{E \in \mesh} L^2(\d E)$, define  scalar products
\begin{gather*}
 (u_h, v_h)_\mesh = \sum_{\elem \in \mesh} \int_\elem u_h v_h \dx, \qquad \langle m_h, \mu_h\rangle_{\partial \mesh} = \sum_{\elem \in \mesh} \langle m_h, \mu_h \rangle_{\partial \elem} = \sum_{\elem \in \mesh} \int_{\partial \elem} m_h  \mu_h \ds, \\
 \text{and }
 \qquad (\vec q_h, \vec r_h)_{\mesh} = \sum_{\elem \in \mesh} \int_\elem  \vec q_h \cdot \vec r_h \dx
\end{gather*}
with induced norms $\|u_h\|_{\mesh}$, $\| m_h \|_{\partial \mesh}$, and $\| \vec q_h \|_{\mesh}$. In this section, we show that  several (classes of) finite element methods verify \Cref{IT:norm}, \Cref{IT:stab} and \Cref{IT:recur}.

\subsection{Mixed methods (MMs)}
%
We consider  mixed methods (MM) where $u_h(\vec y)$ lies in a finite-dimensional subspace $V_h$ of  $ L^2(D)$,   the flux approximation  $\vec q_h(\vec y)$ lies in a finite-dimensional subspace $\mathcal W_h$ of  $ H(\dive,D)$ equipped with norm
\begin{equation} \label{eq:Qh-norm-MM}
 \| \vec q_h(\vec y) \|_{\mathcal W_h} = \| \sqrt{a(\vec y)} \vec q_h(\vec y) \|_{\mesh}.
\end{equation}
The spaces $V_h$ and $\mathcal W_h$ for MM are selected such that there is $\beta > 0$ with
\begin{equation}\label{EQ:inf_sup_mm}
 \sup_{\vec r_h \in \mathcal W_h} \frac{(\nabla \cdot \vec r_h, v_h)_D}{\| \vec r_h \|_{H(\dive,D)}} \ge \beta \| v_h \|_{D} \qquad \text{ holds for all } v_h \in V_h.
\end{equation}
The approximations $\vec q_h$ and $u_h$ satisfy the mixed variational problem
\begin{subequations}\label{EQ:mm_scheme}
\begin{align}
 ( a(\vec y) \vec q_h(\vec y), \vec r_h )_D - ( u_h(\vec y), \nabla \cdot \vec r_h )_D & = G(\vec r_h) && \text{ for all } \vec r_h \in \mathcal W_h, \label{EQ:mm_primal} \\
 (\nabla \cdot \vec q_h(\vec y), v_h)_D & = F(v_h) && \text{ for all } v_h \in V_h,
\end{align}
\end{subequations}
and $F(v_h) = (f, v_h)_D$, $G(\vec r_h) = 0$. Examples of methods in which the conditions on $V_h$ and $\mathcal W_h$ are satisfied include the Raviart--Thomas (RT) method on simplices or $d$-rectangles and Brezzi--Douglas--Marini (BDM) method on simplices. The main result of this subsection is as follows.

\begin{proposition}\label{PROP:assumptions_mm}
 For any MM satisfying the inf-sup condition~\eqref{EQ:inf_sup_mm}, \Cref{IT:norm}, \Cref{IT:stab} and \Cref{IT:recur} hold.
\end{proposition}
\begin{proof}
 It is obvious from~\eqref{eq:Qh-norm-MM} that \Cref{IT:norm} holds with $C_E=1$. By Lemma~\ref{LEM:mm_stab} below, we see that \Cref{IT:stab} holds  with $C_S = 1/\beta$. By Lemma~\ref{LEM:mm_parametric}, \Cref{IT:recur} holds  with $C_R=1$.
\end{proof}

\begin{lemma}\label{LEM:mm_stab}
 The system~\eqref{EQ:mm_scheme} has a unique solution for any $\vec y$ and any $G$ and $F$ in the dual space of $\mathcal W_h$ and $V_h$, respectively. When $G=0$ and $F(v_h) = (f, v_h)_D$, the solution $(u_h(\vec y), \vec q_h(\vec y)) \in V_h \times \mathcal W_h$ satisfies
 \begin{equation*}
  \frac{\beta \| u_h(\vec y) \|_\mesh}{\sqrt{\amax}} \le \| \vec q_h(\vec y) \|_{\mathcal W_h} \le \frac{\sqrt{\amax} \| f \|_\mesh}{\beta}.
 \end{equation*}
\end{lemma}
\begin{proof}
 The well-posedness is a well-known result of the realm of mixed methods. We omit it for brevity and only demonstrate the stability bounds. Testing \eqref{EQ:mm_scheme} with $\vec r_h \leftarrow \vec q_h(\vec y)$ and $v_h \leftarrow u_h(\vec y)$, and adding both equations yields
 \begin{equation*}
  \| \vec q_h(\vec y) \|^2_{\mathcal W_h} = F(u_h(\vec y)) \le \| f \|_\mesh \| u_h(\vec y) \|_\mesh.
 \end{equation*}
 Moreover, \eqref{EQ:inf_sup_mm} and \eqref{EQ:mm_primal} yield
 \begin{align*}
  \beta \| u_h(\vec y) \|_\mesh & \le \sup_{\vec r_h \in \mathcal W_h} \frac{(\nabla \cdot \vec r_h, u_h(\vec y))_D}{\| \vec r_h \|_{H(\dive,D)}} = \sup_{\vec r_h \in \mathcal W_h} \frac{( a(\vec y) \vec q_h(\vec y), \vec r_h )_D}{\| \vec r_h \|_{H(\dive,D)}} \\
  & \le \sqrt{\amax} \| \vec q_h(\vec y) \|_{\mathcal W_h} \sup_{\vec r_h \in \mathcal W_h} \frac{\| \vec r_h \|_D}{\| \vec r_h \|_{H(\dive,D)}} \le \sqrt{\amax} \| \vec q_h(\vec y) \|_{\mathcal W_h},
 \end{align*}
 where we exploited that $\| \vec r_h \|_D \le \| \vec r_h \|_{H(\dive,D)}$. Combining the two equations yields the result.
\end{proof}

\begin{lemma}\label{LEM:mm_parametric}
 We have the parametric regularity bounds
 \begin{equation*}
  \| \der \vec q_h(\vec y) \|_{\mathcal W_h} \le \sum_{\vec 0 \neq \vec m \le \vec \nu} \begin{pmatrix} \vec \nu \\ \vec m \end{pmatrix} \left\| \frac{\derm a(\vec y)}{a(\vec y)} \right\|_{L^\infty(D)} \| \dernum \vec q_h \|_{\mathcal W_h}.
\end{equation*}
\end{lemma}
\begin{proof}
 Differentiating \eqref{EQ:mm_scheme} with respect to the parameter/stochastic variable $\vec y$ yields
 \begin{subequations}\label{EQ:der_mm} 
 \begin{align}
  (a(\vec y) \der \vec q_h(\vec y), \vec r_h )_D - ( \der u_h(\vec y), \nabla \cdot \vec r_h )_D & = G(\vec r_h) && \text{ for all } \vec r_h \in \mathcal W_h, \label{EQ:der_mm_primal}\\
  (\nabla \cdot \der \vec q_h(\vec y), v_h)_D & = F(v_h) && \text{ for all } v_h \in V_h.
 \end{align}
 \end{subequations}
 Clearly, \eqref{EQ:der_mm} has the same shape as \eqref{EQ:mm_scheme}, but $F(v_h) = 0$ and
 \begin{subequations}\label{EQ:der_g}
 \begin{align}
  G(\vec r_h) & = - \sum_{\vec 0 \neq \vec m \le \vec \nu} \begin{pmatrix} \vec \nu \\ \vec m \end{pmatrix} ([\derm a(\vec y)] \dernum \vec q_h(\vec y), \vec r_h )_D \\
  & \le \sum_{\vec 0 \neq \vec m \le \vec \nu} \begin{pmatrix} \vec \nu \\ \vec m \end{pmatrix} \left\| \frac{\derm a(\vec y)}{a(\vec y)} \right\|_{L^\infty(D)} \| \dernum \vec q_h \|_{\mathcal W_h} \| \vec r_h \|_{\mathcal W_h}.
 \end{align}
 \end{subequations}
 Substituting $\der \vec q_h (\vec y)$ for $\vec r_h$ and $\der u_h(\vec y)$ for $v_h$ in \eqref{EQ:der_mm} and adding, the result follows.
\end{proof}

\subsection{Hybridizable and embedded discontinuous Galerkin (HDG and EDG)}\label{SEC:hdg}
%
The hybridizable discontinuous Galerkin (HDG) method identifies $u_h(\vec y) \in V_h \subset H^1(\mesh)$, its flux $\vec q_h(\vec y) \in \mathcal W_h \subset H(\dive, \mesh)^d$ with norm
\begin{equation}\label{EQ:hdg_norm}
 \| \vec q_h(\vec y) \|_{\mathcal W_h} = \| \sqrt{a(\vec y)} \vec q_h(\vec y) \|_{\mesh} + \|\sqrt\tau (u_h \corr{(\vec y)} - m_h \corr{(\vec y)} ) \|_{\partial \mesh},
\end{equation}
and $m_h\corr{(\vec y)} \in M_h \subset L^2(\faces)$ with $m_h\corr{(\vec y)}|_{\partial D} = 0$ such that
\begin{subequations}\label{EQ:hdg_scheme}
\begin{align}
 ( a(\vec y) \vec q_h(\vec y), \vec r_h )_\mesh - ( u_h(\vec y), \nabla \cdot \vec r_h )_\mesh + \langle m_h (\vec y), \vec r_h \cdot \normal \rangle_{\partial \mesh} & = G(\vec r_h) && \text{for all } \vec r_h \in \mathcal W_h, \label{EQ:hdg_primal} \\
 (\nabla \cdot \vec q_h(\vec y), v_h)_\mesh + \langle u_h(\vec y) - m_h(\vec y), \tau v_h \rangle_{\partial \mesh} & = F(v_h) && \text{for all } v_h \in V_h,\\
 \langle \vec q_h(\vec y) \cdot \normal + \tau ( u_h(\vec y) - m_h(\vec y) ), \mu_h \rangle_{\partial \mesh} & = 0 && \text{for all } \mu_h \in M_h,\label{EQ:mass_conservation}
\end{align}
\end{subequations}
where we choose the spaces $V_h$ and $\mathcal W_h$, such that there is $\beta > 0$ with
\begin{equation}\label{EQ:inf_sup_hdg}
 \sup_{\vec r_h \in \mathcal W_h } \frac{(\nabla \cdot \vec r_h, v_h)_\elem}{\| \vec r_h \|_{H(\dive,\elem)}} \ge \beta \| v_h \|_\elem \qquad \text{ holds for all } v_h \in V_h \text{ and all } \elem \in \mesh,
\end{equation}
and the parameter $\tau = 0$. Alternatively, we choose the parameter $\tau > 0$ and can freely choose the spaces $V_h$ and $\mathcal W_h$. The former setting includes the hybridized Raviart--Thomas (RT-H) method (on simplices and \(d\)-rectangles) and the hybridized Brezzi--Douglas--Marini (BDM-H) method (on simplices). At the same time, the latter approach covers the hybridizable local discontinuous Galerkin scheme (LDG-H) and the class of embedded discontinuous Galerkin (EDG) schemes. Again, $F(v_h) = (f, v_h)_D$, $G(\vec r_h) = 0$.

Notably, \eqref{EQ:hdg_scheme} can be statically condensed to form a symmetric positive definite system of linear equations; see \cite[Thm.\ 2.1]{CockburnGL09}, and we have the following relation to the mixed method:

\begin{proposition}\label{LEM:hdg_is_mm}
 The RT-H and BDM-H methods produce the same approximations as the RT and BDM methods with \( m_h = u_h \) on the mesh faces, respectively. That is, \Cref{IT:norm}, \Cref{IT:stab}, and  \Cref{IT:recur} hold.
\end{proposition}
\begin{proof}
 This standard result uses that condition \eqref{EQ:mass_conservation} implies that $\vec q_h \in H(\dive, D)$. The constants can be made explicit as performed in \Cref{PROP:assumptions_mm}.
\end{proof}

The \emph{inf-sup} condition \eqref{EQ:inf_sup_hdg} allows us to bound \( \| u_h(\vec y) \| \lesssim \sqrt{\amax} \| \vec q_h(\vec y) \|_{\mathcal W_h} \), which is needed to show show \Cref{IT:stab}, see the proof of \Cref{LEM:mm_stab}. For HDG methods, we assume a similar inequality, i.e.,
\begin{equation}\label{EQ:stab_u}
 \| u_h \|_\mesh \le C_B \sqrt{\overline a(\vec y)} \| \vec q_h \|_{\mathcal W_h},
\end{equation}
which can (in general) be established from \eqref{EQ:hdg_primal} and \eqref{EQ:inf_sup_hdg}, or from a generalization of Poincar\'e's inequality, or investigating the eigenvalues of HDG/EDG methods. If we use the BDM-H, RT-H, LDG-H, or EDG methods on simplicial or quadrilateral meshes, \eqref{EQ:stab_u} is a direct consequence of \cite[Thm. B.3]{LuMR23} or (less-direct) of \cite{Brenner03}. Relation \eqref{EQ:stab_u} is also vital in the theory of multigrid methods for HDG schemes, see \cite[(LS2) \& (LS6), shown in Sect.\ 6]{LuRK22a}. Moreover, \cite{Brenner03} allows to extend the result for LDG-H and EDG to certain non-simplicial/non-quadrilateral meshes.

\begin{proposition}
 For any HDG method, \Cref{IT:norm} and \Cref{IT:recur} hold. If additionally, \eqref{EQ:stab_u} holds, then also \Cref{IT:stab} holds.
\end{proposition}
\begin{proof}
 The validity of \Cref{IT:norm} follows directly from \eqref{EQ:hdg_norm} with \( C_E = 1 \), and \Cref{IT:recur} is shown with \( C_R = 1 \) in \Cref{LEM:hdg_parametric}. Finally, \Cref{IT:stab} is demonstrated in \Cref{LEM:stability_hdg} with \( C_S = C_B \).
\end{proof}

\begin{lemma}\label{LEM:stability_hdg}
 If \eqref{EQ:stab_u} holds, the solution of the HDG method \eqref{EQ:hdg_scheme} satisfies
 \begin{equation*}
  \| \vec q_h(\vec y) \|_{\mathcal W_h} \le C_S \sqrt{\amax} \| f \|_\mesh.
 \end{equation*}
\end{lemma}
\begin{proof}
 Setting $\vec r_h \leftarrow \vec q_h(\vec y)$, $v_h \leftarrow u_h(\vec y)$, $\mu_h \leftarrow -m_h(\vec y)$ and summing the equations in \eqref{EQ:hdg_scheme} yields
 \[
  \| \vec q_h(\vec y) \|^2_{\mathcal W_h} = F(u_h(\vec y)) \le \| f \|_\mesh \| u_h (\vec y) \|_\mesh \lesssim \| f \|_\mesh \sqrt{\amax} \| \vec q_h(\vec y) \|_{\mathcal W_h}
 \]
 where the hidden constant in `\( \lesssim \)' is \( C_S = C_B \).
\end{proof}

\begin{lemma}\label{LEM:hdg_parametric}
 We have the parametric regularity bound
 \begin{equation*}
  \| \der \vec q_h(\vec y) \|_{\mathcal W_h} \le \sum_{\vec 0 \neq \vec m \le \vec \nu} \begin{pmatrix} \vec \nu \\ \vec m \end{pmatrix} \left\| \frac{\derm a(\vec y)}{a(\vec y)} \right\|_{L^\infty(D)} \| \dernum \vec q_h(\vec y) \|_{\mathcal W_h}.
 \end{equation*}
\end{lemma}
\begin{proof}
 Differentiating \eqref{EQ:hdg_scheme} with respect to the parameter/stochastic variable $\vec y$ yields
 \begin{subequations}\label{EQ:der_hdg}
 \begin{align}
  (a(\vec y) \derm \vec q_h(\vec y), \vec r_h )_\mesh - ( \derm u_h(\vec y), \nabla \cdot \vec r_h )_\mesh + \langle \derm m_h (\vec y), \vec r_h \cdot \normal \rangle_{\partial \mesh} & = G(\vec r_h), \\
  (\nabla \cdot \derm \vec q_h(\vec y), v_h)_\mesh + \langle \derm u_h(\vec y) - \derm m_h(\vec y), \tau v_h \rangle_{\partial \mesh} & = F(v_h),\\
  \langle \derm \vec q_h(\vec y) \cdot \normal + \tau ( \derm u_h(\vec y) - \derm m_h(\vec y) ), \mu_h \rangle_{\partial \mesh} & = 0.
  \end{align}
 \end{subequations}
 Clearly, \eqref{EQ:der_hdg} has the same form as \eqref{EQ:hdg_scheme}, but we have that $F(v_h) = 0$ and $G(\vec r_h)$ as in \eqref{EQ:der_g}. Combining the techniques in the proofs of \Cref{LEM:mm_parametric,LEM:stability_hdg} yields the result.
\end{proof}

\begin{remark}[Hybrid high-order methods]
 The above considerations can also be extended to hybrid-high-order methods, for which the stabilizing boundary integral (the one involving \( \tau \)) is replaced by an abstract (stabilizing) bilinear form and \eqref{EQ:stab_u} follows from \cite[(HM1), (HM3) or (HM6)]{DiPietroDKMR24}.
\end{remark}

\subsection{Continuous finite elements}
%
Continuous (also classic, conforming, or standard) finite elements approximate the solution by $u_h \in U_h$ \corr{(where \( U_h  \subset H^1(D) \) is an element-wise polynomial finite element space)} with
\begin{equation}\label{EQ:fe_disc}
 \int_D a^{-1}(\vec y) \nabla u_h(\vec y) \cdot \nabla v_h \dx = \int_D f v_h \dx \qquad \text{ for all } v_h \in U_h.
\end{equation}

Setting \(\vec q_h = - a^{-1}(\vec y) \nabla u_h(\vec y) \in \mathcal W_h = L^2(\Omega) \) in the interior of all $\elem \in \mesh$ and \( \| \vec q_h(\vec y) \|_{\mathcal W_h} = \| \sqrt{a(\vec y)} \vec q_h(\vec y) \|_{L^2(D)} \) allows us to employ our general framework---note that in this case $\mathcal W_h$ is \emph{not} finite dimensional.

Here, \(C_S > 0\) is the Friedrichs--Poincar\'e constant for which the relation \( \| u_h \|_{L^2(D)} \le C_S \| \nabla u_h \|_{L^2(D)}\) holds, since
\[ \| \vec q_h(\vec y) \|^2_{\mathcal W_h} = \int_D f u_h(\vec y) \dx \le C_S \sqrt{\overline a(\vec y)} \| f \|_{L^2(D)} \| \vec q_h(\vec y) \|^2_{\mathcal W_h}. \]

The constant \(C_R = 1\), since differentiating \eqref{EQ:fe_disc} implies gives a zero on the right-hand side (\( v_h \) is independent of \( \vec y \)), and\\
\begin{align*}
 \| \der \vec q_h(\vec y) \|^2_{\mathcal W_h} & = - \sum_{\vec 0 \neq \vec m \le \vec \nu} \begin{pmatrix} \vec \nu \\ \vec m \end{pmatrix} ([\derm a(\vec y)] \dernum \vec q_h(\vec y), \der \vec q_h(\vec y) )_D \\
 & \le \sum_{\vec 0 \neq \vec m \le \vec \nu} \begin{pmatrix} \vec \nu \\ \vec m \end{pmatrix} \left\| \frac{\derm a(\vec y)}{a(\vec y)} \right\|_{L^\infty(D)} \| \dernum \vec q_h(\bsy) \|_{\mathcal W_h} \| \der \vec q_h(\vec y) \|_{\mathcal W_h}.
\end{align*}

\section{Numerical experiments}\label{SEC:numerics}
%
We consider equation \eqref{EQ:base_pde} with $f(\vec x) = x_1$ in $D = (0,1)^2$ and investigate the errors in the means of the numerical approximations to the unknown $u(\vec y$, its gradient $\nabla u(\vec y) = a(\vec y) \vec q(\vec y)$, and its flux $\vec q(\vec y)$ in the subdomain $\tilde D = (0.2,0.8)^2$.

For the affine case, we set $U = [-\tfrac{1}{2},\frac{1}{2}]^{\mathbb N}$ and truncate the series expansion for the input random coefficient into $s=100$ terms, i.e.,
\begin{equation*}
 a_\textup{affine}(\vec x,\vec y) = 5 + \sum_{j=1}^{100} \frac{\xi(y_j)}{(k_j^2 + \ell_j^2)^{1.3}} \sin(k_j \pi x_1) \sin(\ell_j \pi x_2),
\end{equation*}
where $(k_j, \ell_j)_{j \ge 1}$ is an ordering of elements of $\mathbb Z_+ \times \mathbb Z_+$ such that the sequence $(\|\psi_j\|_{L^\infty(D)})_{j \ge 1}$ is not increasing. For the affine Gevrey-$\sigma$ case, we set $\sigma = 1.25$ and choose mapping $\xi$ as 
\begin{equation*}
 \xi(t) = \exp \left( \frac{-1}{(t + 0.5)^\omega} \right), \qquad \text{ with } \qquad \omega =\frac{1}{\sigma - 1} \iff \sigma=1+\frac{1}{\omega},
\end{equation*}
while $\xi$ is the identity in the (classical) affine case.

In the lognormal case, we define $U = \mathbb R^{\mathbb N}$ and consider the dimensionally truncated coefficient with $s=100$ terms
\begin{equation*}
 a_\textup{lognormal}(\vec x, \vec y) = \exp\left( \sum_{j=1}^{100} \frac{\xi(y_j)}{(k_j^2 + \ell_j^2)^{1.3}} \sin(k_j \pi x_1) \sin(\ell_j \pi x_2) \right)
\end{equation*}
with an analogously defined sequence $(k_j, \ell_j)_{j \ge 1}$. In this case, we have that $\|\psi_j\|_{L^\infty(D)}\sim j^{-1.3}$ and the expected convergence independently of the dimension is $\mathcal O(n^{-0.8+\varepsilon})$, $\varepsilon>0$. Again, $\xi$ is the identity in the (classical) lognormal case, and for the lognormal Gevrey case, we choose
\[ \xi(t)= \operatorname{sign}(t)\exp(-t^{-\omega}) \qquad \text{ with } \qquad \omega =\frac{1}{\sigma - 1} \]
(cf., e.g.,~\cite[p.~16]{qi1996general}) and the convention that $\operatorname{sign}(0) = 0$.

We use the off-the-shelf generating vector~\cite[lattice-32001-1024-1048576.3600]{lattice} with a total amount of $2^m$, $m = 2, \dots, 14$  cubature points with $R=16$ random shifts. Although this generating vector has not been obtained using the CBC algorithm with the weights derived in \cite{GrahamKNSSS15} and Theorem \ref{thm:qmcweight2}, we found this off-the-shelf generating vector to perform well, yielding the \corr{theoretically expected} rate of convergence. Thus, this confirms our analytical findings without excluding that there might be even better lattice rules. In practice, the off-the-shelf lattice rule performs similarly to tailored lattice rules. The finite element discretization uses the lowest order Raviart--Thomas discretization on an unstructured simplicial mesh. The file that needs to be run to reproduce our results can be found in 
\cite{KaarniojaRG25}; it uses NGSolve \cite{Schoberl14}.

\begin{figure}[ht!]
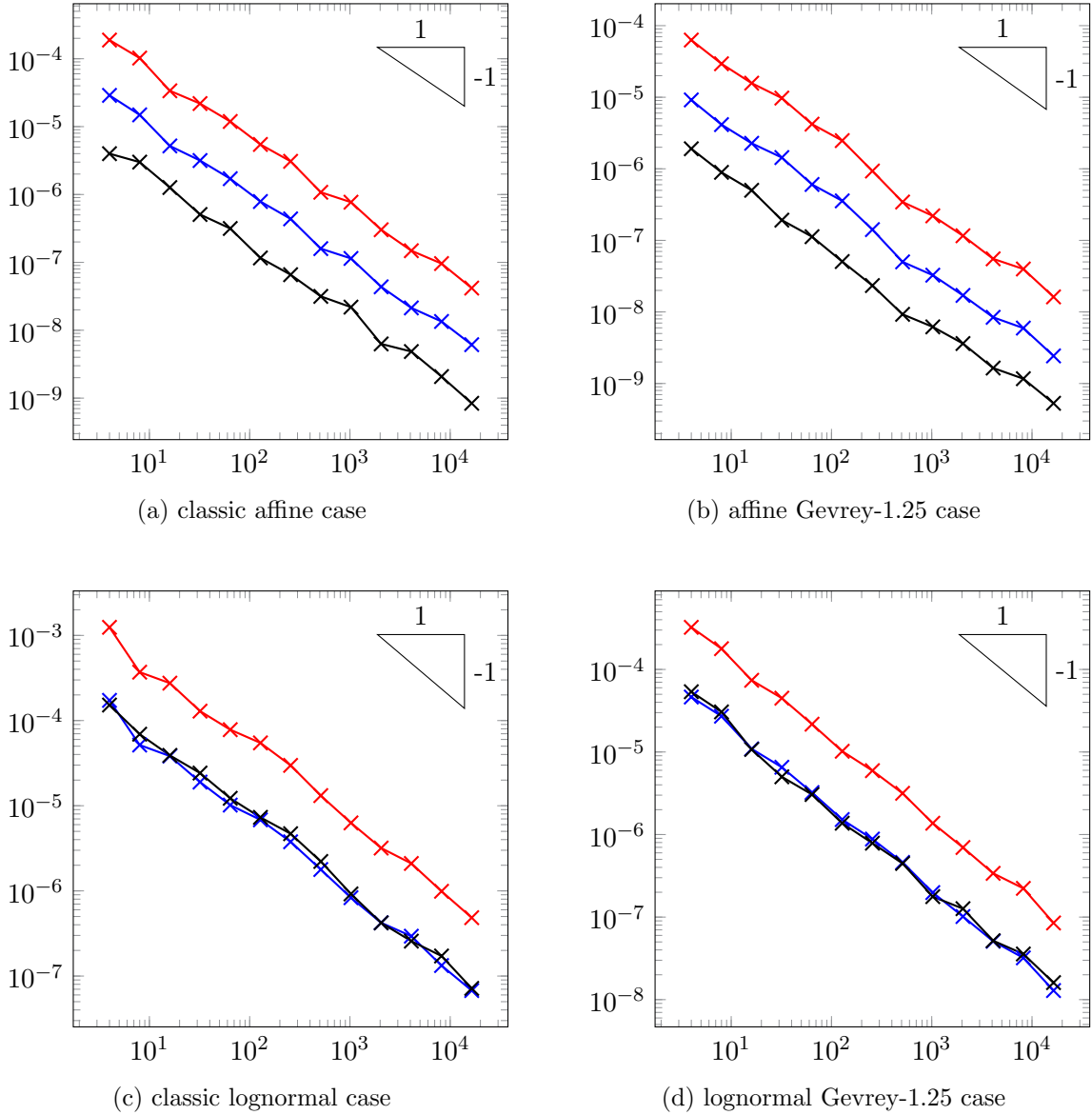
\centering
 \begin{subfigure}[t]{0.48\textwidth}\centering
  \includegraphics[width=.95\textwidth]{results/affine}
  \caption{classic affine case}
 \end{subfigure}\hfill
 \begin{subfigure}[t]{0.48\textwidth}\centering
  \includegraphics[width=.95\textwidth]{results/gevrey-affine}
  \caption{affine Gevrey-1.25 case}
 \end{subfigure}\\[3ex]
 \begin{subfigure}[t]{0.48\textwidth}\centering
  \includegraphics[width=.95\textwidth]{results/lognormal}
  \caption{classic lognormal case}
 \end{subfigure}\hfill
 \begin{subfigure}[t]{0.48\textwidth}\centering
  \includegraphics[width=.95\textwidth]{results/gevrey-lognormal}
  \caption{lognormal Gevrey-1.25 case}
 \end{subfigure}
 \caption{Root mean square error (ordinates) of the mean value of $u_h(\vec y)$ in blue, its gradient $-a(\vec y) \vec q_h(\vec y)$ (red), and its flux $\vec q_h(\vec y)$ (black). The abscissas illuminate the number of QMC quadrature points per random shift.}\label{FIG:numerics}
\end{figure}

In the (classic) affine (Figure \ref{FIG:numerics}, (A)), we observe that the QMC error for the gradient (red) is the largest and decays with a rate of 1.01, while the QMC error for the numerical flux is the smallest and decays with a rate of 1.02. The QMC error for the unknown $u_h(\vec y)$ lies in between and decays with a rate of 1.01. The affine Gevrey-$1.25$ case (Figure \ref{FIG:numerics}, (B)) behaves very similarly with decay rates of 1.00, 0.99, and 1.00 for the respective quantities.

For the lognormal case (Figure \ref{FIG:numerics}, (C)), we observe that the root mean square error for the primal unknown $u_h(\vec y)$ decreases with a rate of $0.90$ and the root mean square error for the gradient $-a(\vec y) \vec q_h(\vec y)$ decreases with a rate of $0.90$. Thus, these two quantities seem to decrease \corr{at the theoretically expected rate}. However, the convergence rate for the flux unknown $\vec q_h(\vec y)$ is slightly better and lies at $0.91$. We also see that the root mean square errors for the primal unknown and the flux are similar. In contrast, the respective error for the gradient is larger and highlights that considering only $2^i$ with $i \ge 8$ points yields convergence rates of 0.94 for $u_h(\vec y)$, 0.96 for $\nabla u_h(\vec y)$, and 0.98 for $\vec q_h(\vec y)$ indicating that our QMC error bounds are accurate for large sample sizes. The lognormal Gevrey-1.25 case (Figure \ref{FIG:numerics}, (D)) converges with orders 0.98 for gradient and primal unknown and with order 0.96 for the flux unknown.

\section{Conclusions}
%
In this work, we derived parametric regularity bounds for the diffusion equation with a parameter inversely proportional to Gevrey regular parametrizations of the input random field $a$. We observe that mixed and mixed hybrid methods are natural discretizations for such problems and that the parametric regularity estimates are mainly performed in terms of the flux unknown $\vec q_h(\vec y)$, implying that quantities of interest that directly rely on $\vec q_h$ should be approximated best while quantities of interest that rely on $u_h(\vec y)$ or $\nabla u_h(\vec y)$ have an additional continuity constant in the analysis, and should converge at the same rate but with larger constant. Our numerical experiments confirm this observation.

Hybrid and mixed methods substantially improve the QMC results since they correctly represent more aspects of physics and lay the foundation of the above general framework to analyze QMC finite element methods. In the future, we would like to use the additional information from the unknown flux to build adaptive QMC methods that spatially resolve the PDE problem in each quadrature node just as necessary.

A limitation of our framework is numerical methods that need a rich enough stabilization to approximate the quantity of interest stably, such as (interior penalty) discontinuous Galerkin or discontinuous Petrov--Galerkin methods. Such methods typically violate \Cref{IT:recur} and need individualized analysis techniques; see \cite{KaarniojaR24} for discontinuous Galerkin.

\begin{remark}[Possible generalizations]\
 \begin{enumerate}
  \item \textbf{More general equations.} A careful revision of the arguments in Sections \ref{SEC:qmc} and \ref{SEC:para_reg} reveals that \eqref{EQ:base_pde} is never used in the QMC analysis. Thus, we could conduct the same analysis if Assumptions \ref{IT:stab}--\ref{IT:gev} hold for any other (partial differential) equation.
  \item \textbf{Sparse grids.} The regularity bounds in Section \ref{SEC:para_reg} can also be used to analyze the convergence properties of sparse grid methods. Thus, our framework also covers sparse grids if one replaces Theorems \ref{THM:bounded} and \ref{thm:lognormalqmcerror} with their respective sparse grid versions.
  \item \textbf{Quadratic quantities of interest:} If $\|\partial_{\setu}\boldsymbol q_h\|_{L^2(D)}\leq C(|\setu|!)^{\sigma}\boldsymbol b_{\setu}$, then~\eqref{eq:quadraticfun} can be sharpened as
  \begin{align*}
   \bigg|\frac{\partial^{|\setu|}}{\partial \bsy_{\setu}}J(\boldsymbol q_h(\bsy))\bigg|&\leq C^2\boldsymbol b_{\setu}\sum_{\setv\subseteq\setu}(|\setv|!)^{\sigma}((|\setu|-|\setv|)!)^\sigma=C^2\boldsymbol b_{\setu}\sum_{\ell=0}^{|\setu|}(\ell!)^{\sigma}((|\setu|-\ell)!)^{\sigma}\\
   & \leq C^2\boldsymbol b_{\setu}\sum_{\ell=0}^{|\setu|}(|\setu|!)^{\sigma} =C^2\boldsymbol b_{\setu}(|\setu|!)^{\sigma}(|\setu|+1)\leq C^2\boldsymbol b_{\setu}((|\setu|+1)!)^{\sigma},
  \end{align*}
  where we used the fact that $\ell!(|\setu|-\ell)!\leq |\setu|!$ and especially $(\ell!(|\setu|-\ell)!)^\sigma\leq (|\setu|!)^\sigma$ for $\sigma\geq 1$. This estimate can be used to shorten and improve the estimates of \( t_{\setu} \) in the proofs of Theorems \ref{thm:thm1} and \ref{thm:qmcweight2}: Repeating the subsequent arguments in the respective proofs allows to obtain the convergence results with \( r = 1 \) instead of \(r = 2 \) for quadratic quantities of interest.
 \end{enumerate}
\end{remark}

\bibliographystyle{siamplain}
\bibliography{qmc_mixed}

\appendix

\section{More details} \label{sec:more-details}

\subsection{Proof of \eqref{eq:qmc-expectation}}
%
Using a translational change of variable in the definition~\eqref{EQ:eval_int} of $I_s^\vphi$, we have, for any 
fixed $k$,
\[
 I_s^{\varphi} = \int_{(0,1)^s}F(\Phi^{-1}(\boldsymbol t))\,{\rm d}\boldsymbol t  = \int_{(0,1)^s} F(\Phi^{-1}(\{\boldsymbol t_k+\boldsymbol\Delta_1\})) \; \textup d \boldsymbol\Delta_1.
\]
Averaging over~$k$, 
\[
 I_s^{\varphi} = \int_{(0,1)^s} \frac{1}{n} \sum_{k=1}^n F(\Phi^{-1}(\{\boldsymbol t_k+\boldsymbol\Delta_r\})) \; \textup d \boldsymbol\Delta_1 = \EE [ Q_1^\vphi],
\]
and we obtain the second equality of \eqref{eq:qmc-expectation} with $r=1$.  It follows for any~$r$ since $\bd \Delta_r$ are i.i.d.\ and $\EE [ Q_1^\vphi] = \EE [ Q_r^\vphi]$.  The first equality of~\eqref{eq:qmc-expectation} follows from
\begin{align*}
 \EE[ \Qmc ] & = \int_{(0,1)^s \times \dots \times (0,1)^s} \left( \frac 1 R \sum_{r=1}^R Q_r^\vphi \right) \; \textup d (\boldsymbol\Delta_1 \otimes \dots \otimes \boldsymbol\Delta_R) \\
 & = \frac 1 R    \sum_{r=1}^R \left( \int_{(0,1)^s} Q_r^\vphi \; \textup d \bd\Delta_r\right) \; \left(\prod_{r'\ne r} \int_{(0,1)^s}    \textup d \boldsymbol\Delta_{r'}\right) = \frac 1 R    \sum_{r=1}^R \EE[ Q_r^\vphi],  
\end{align*}
because $\int_{(0,1)^s} \textup d \bd \Delta_{r'} = 1$.  Since the values of $\EE[ Q_r^\vphi]$ are equal for all $r$, \eqref{eq:qmc-expectation} follows.

\subsection{Proof of \eqref{eq:qmc-variance}}
%
The variance in \eqref{eq:qmc-variance} equals
\begin{align*}
 \EE\big[\left(I_s^\varphi -\Qmc \right)^2\big] &=  R^{-2}\, \EE\Big[\Big( \sum_{r=1}^R (I_s^\varphi -  Q_r^\vphi) \Big)^2\Big] \\
 & = R^{-2} \sum_{r=1}^R \EE\big[( I_s^\varphi -  Q_r^\vphi)^2\big] + R^{-2} \sum_{r' \ne r} \EE\big[ ( I_s^\varphi -  Q_r^\vphi)( I_s^\varphi -  Q_{r'}^\vphi) \big] \\
 &= \frac 1 R  \EE\big[( I_s^\varphi -  Q_r^\vphi)^2\big] + \frac{1}{R^2}\sum_{r\ne r'}( I_s^\varphi)^2 -  2 I_s^\vphi \EE [Q_r^\vphi] + \EE[Q_r^\vphi Q_{r'}^\vphi], 
\end{align*}
where we have used equality of all summands of the first sum.  By \eqref{eq:qmc-expectation}, $I_s^\vphi \EE[ Q_r^\vphi ] = (I_s^\vphi)^2$, while by the independence of $\bd\Delta_r$, we have $\EE[Q_r^\vphi Q_{r'}^\vphi] = \EE[Q_r^\vphi]^2 = (I_s^\vphi)^2$. Hence the last term vanishes, thus proving \eqref{eq:qmc-variance}.


\end{document}